\documentclass[11pt,amsfonts]{amsart}
\usepackage{a4wide}
\usepackage{todonotes}
\usepackage[english]{babel}
\usepackage{inputenc}
\usepackage{amsmath,amsthm,amsfonts,amssymb,amscd}
\usepackage{mathpazo}
\usepackage{amssymb,amsfonts}
\usepackage[all,arc]{xy}
\usepackage{enumerate}
\usepackage{tgtermes}
\usepackage{marginnote}
\usepackage{mathtools}
\usepackage{tikz,amsmath, amssymb,bm,color}
\usetikzlibrary{shapes,arrows}
\usetikzlibrary{calc}
\usetikzlibrary{snakes}
\usepackage{amsthm,thmtools,xcolor}

\usepackage[all]{xy}
\usepackage{tikz-cd}
\usetikzlibrary{cd}

\newcommand{\R}{\mathbb R}

\newcommand{\hil}{\mathcal{H}}

\newcommand{\la}{\lambda}

\theoremstyle{plain}
\newtheorem{thm}{Theorem}[section]

\newtheorem{prop}{Proposition}[section]
\newtheorem{lem}{Lemma}[section]
\newtheorem{thmp}{Theorem}[section]

\newtheorem{lemx}{Lemma}

\theoremstyle{definition}
\newtheorem{defn}{Definition}[section]

\theoremstyle{remark}
\newtheorem{rem}{\textit{Remark}}[section]

\numberwithin{equation}{section}
\usepackage{color}

\usepackage[colorlinks = true,
linkcolor = blue,
urlcolor  = blue,
citecolor = blue,
anchorcolor = blue]{hyperref}
\usepackage{hyperref}
\makeatletter
\let\c@equation\c@thm
\makeatother

\newcommand\underrel[3][]{\mathrel{\mathop{#3}\laaimits_{%
			\ifx c#1\relax\mathclap{#2}\else#2\fi}}}
\title[extended Schr\"{o}dinger- Benjamin-Ono system]{Well-posedness for the extended Schr\"{o}dinger-Benjamin-Ono system}

\author{F. Linares }

\address{Instituto de Matem\'{a}tica Pura e Aplicada-IMPA, Estrada Dona Castorina, 110
	Jardim Bot\^{a}nico,
	Rio de Janeiro, RJ - Brazil}
\email{linares@impa.br}

\author{ A.J.  Mendez }

\address{Pontificia Universidad Cat\'{o}lica de Valpara\'{\i}so, Blanco Viel 596, Cerro Bar\'{o}n, Valpara\'{\i}so, Chile}
\email{argenis.mendez@pucv.cl}

\author{D. Pilod}

\address{Department of Mathematics, University of Bergen, Postbox 7800, 5020 Bergen, Norway}

\email{Didier.Pilod@uib.no}


\date{\today}
\commby{}
\begin{document}

\maketitle

\begin{center} 
\small{\emph{Dedicated to Carlos E. Kenig for his $70^{th}$ birthday, \\ 
with friendship and admiration}}.
\end{center}
	
\begin{abstract}
In this work we prove that  the initial value problem  associated to the Schr\"{o}dinger-Benjamin-Ono type system 
\begin{equation*}
\left\{ \begin{array}{ll}
	\mathrm{i}\partial_{t}u+ \partial_{x}^{2} u=  uv+ \beta u|u|^{2}, \\
	\partial_{t}v-\mathcal{H}_{x}\partial_{x}^{2}v+ \rho v\partial_{x}v=\partial_{x}\left(|u|^{2}\right)\\
	u(x,0)=u_{0}(x), \quad v(x,0)=v_{0}(x),
	\end{array} 
\right.
\end{equation*}
with $\beta,\rho \in \mathbb{R}$ is locally well-posed for initial data $(u_{0},v_{0})\in H^{s+\frac12}(\mathbb{R})\times H^{s}(\mathbb{R})$ for $s>\frac54$. 

Our method of proof relies on energy methods and compactness arguments. However, due to the lack of symmetry of the nonlinearity, the usual energy has to be modified to cancel out some bad terms appearing in the estimates. Finally, in order to lower the regularity below the Sobolev threshold $s=\frac32$, we employ a refined Strichartz estimate introduced in the Benjamin-Ono setting by Koch and Tzvetkov, and further developed by Kenig and Koenig. 
\end{abstract}


\section{Introduction}

\subsection{The Schr\"odinger-Benjamin-Ono equation}
We are interested in the study of the initial value problem (IVP) associated to the following system of nonlinear dispersive equation, i.e.
\begin{equation}\label{sbo}
\begin{cases}
\mathrm{i}\partial_t u+\partial_x^2 u =  uv+\beta|u|^2u,\hskip40pt x\in\R,\;t>0,\\
\;\partial_tv-\hil \partial_x^2 v+ \rho v\partial_xv=\partial_x(|u|^2),\\
u(x, 0)= u_0(x), \hskip10pt v(x,0)=v_0(x),
\end{cases}
\end{equation}
where $u=u(x,t)$ is a complex valued function, $v=v(x,t)$ is a real valued function, the parameters $\beta, \rho\in\R$, and $\hil$ denotes the Hilbert transform, defined on the line as
\begin{equation}\label{hilbert}
\hil f(x)= {\rm p.v.} \,\frac{1}{\pi}\int \frac{f(y)}{x-y}\, dy.
\end{equation}
When $\rho \neq 0$, we will refer to the system in \eqref{sbo} as the \emph{extended Schr\"odinger-Benjamin-Ono system}.

\medskip

The system \eqref{sbo} appears as a particular case (under appropriate transformations) of the more general system describing the interaction phenomenon between long waves and short waves under a weakly coupled nonlinearity, 
\begin{align}\label{long-short}
\begin{cases}
i\partial_t S+ic_S \partial_x S+\partial_x^2 S=\alpha SL+\gamma\vert S\vert^2S, & c_s,\alpha,\gamma\in\R,
\\ \partial_t L+c_L \partial_x L+\nu P(D_x)L+\lambda \partial_x L^2=\beta \partial_x \vert S\vert^2, & c_L,\nu,\lambda,\beta\in\R,
\end{cases}
\end{align}
where $S=S(x,t)$ is a complex-valued function representing the short wave, $L=L(x,t)$ is a real-valued function representing the long wave and $P(D_x)$ is a differential operator with constant coefficients. This system has received considerable attention because of the vast variety of physical situations in which it arises. (See \cite{BeBu}, \cite{DjRe},
\cite{FuOi}, \cite{Gr}, \cite{Ma}, \cite{NHMI}, \cite{SaYa} ).


Of particular interest is a model for the motion of two fluids under capillary-gravity waves in a deep water flow that was derived in \cite{FuOi} and it corresponds to
\vspace{-2mm}
\begin{equation*}
P(D_x)=D_x^1\partial_x=\hil\partial_x^2, \quad \nu\neq 0, \quad c_s=c_L=\gamma=\lambda=0, \quad \alpha,\,\beta>0,
\end{equation*}
in \eqref{long-short}, that is,
\begin{equation}\label{deep-water}
\begin{cases}
i\partial_tS+\partial_x^2S =  \alpha SL,\\
\;\partial_t L+\nu\hil \partial_x^2 L=\beta \partial_x(|S|^2).
\end{cases}
\end{equation}
which is referred to in the literature as the \emph{Schr\"odinger-Benjamin-Ono system}. 

\medskip 

We also shall comment that smooth solutions of the system in \eqref{sbo} satisfy the 
following conserved quantities:
\begin{equation}
\mathcal{E}_{1}(t):=	 \int_{\mathbb{R}}v(x,t)\,dx=\mathcal{E}_{1}(0),
 \end{equation}
\begin{equation}\label{conserv1}
\mathcal{E}_{2}(t):=	 \int_{\mathbb{R}}|u(x,t)|^{2}\,dx=\mathcal{E}_{2}(0),
\end{equation}
\begin{equation}\label{conserv2}
\mathcal{E}_{3}(t):=	\mathrm{Im}\int_{\mathbb{R}}u(x,t)\overline{\partial_{x}u(x,t)}\,dx+\frac{1}{2}\int_{\mathbb{R}}v^{2}(x,t)\,dx
=\mathcal{E}_{3}(0),
\end{equation}
and
\begin{equation}
\begin{split}
\mathcal{E}_{4}(t)\!\!&:=\frac{1}{2}\int_{\mathbb{R}}(D_{x}^{\frac{1}{2}}v(x,t))^{2}dx -\frac{\rho}{6}\int_{\mathbb{R}}v^{3}(x,t)dx+\int_{\mathbb{R}}v(x,t)|u(x,t)|^{2}dx\\
&\quad+\frac{\beta}{2}\int_{\mathbb{R}} |u|^{4}(x,t)\,dx+\int_{\mathbb{R}}|\partial_{x}u(x,t)|^{2}\,dx=\mathcal{E}_{4}(0),
\end{split}
\end{equation}
where the $D^{\frac12}_x$ operator is defined in the notation at the end of the introduction.
As far as we know this system is not completely integrable in contrast with the cubic nonlinear Schr\"odinger
and the Benjamin-Ono equations which are coupled in system \eqref{sbo}.

\medskip

\subsection{Well-posedness results}

Our main goal in this paper is to establish a first result regarding local well-posedness for  the IVP \eqref{sbo}.  To our knowledge  there is no such theory in this case when $\rho \neq 0$.  Before describing our results we shall comment what has been done for the system \eqref{deep-water}, corresponding to \eqref{sbo} with $\rho=\beta=0$, to motivate our analysis.

Regarding well-posedness for the IVP associated to system \eqref{deep-water}, for initial data $(S(0), L(0))\in H^{s}(\R)\times H^{s-\frac12}(\R)$, local well-posedness was proved for $s\ge 0$ in the case $|\nu|\neq 1$ in \cite{bop1}, and then for $s>0$ in the case the case $|\nu| = 1$  in \cite{P}. Finally, we refer to \cite{Dom} for sharp local well-posedness results in $H^s(\mathbb R) \times H^{s'}(\mathbb R)$  in the non-resonant case $|\nu|\neq 1$ and under some restrictions on the parameters $(s,s')$, which can be decoupled from $(s,s-\frac12)$. To obtain the aforementioned results, the restriction Fourier method introduced in \cite{Bo} and a fixed point argument were used. 
We shall remark that solutions of the system \eqref{long-short} satisfy conserved quantities (see \cite{bop1} for instance) that allow to extend the local solutions
globally in the energy space under some restrictions on the parameters $\nu, \alpha$, and $\beta$ (see \cite{bop1} and \cite{P}).  In  \cite{amp}, the global theory was extended to $H^{s}(\R)\times H^{s-\frac12}(\R)$ for $s\ge 0$ and any $\nu \neq 0$. System \eqref{deep-water} was also considered in the periodic setting (see \cite{amp} for the well-posedness theory and \cite{oh} for the construction of invariant measures).

Notice that the weak nonlinear interaction $L\partial_xL$ is missing in \eqref{deep-water}. Nowadays it is known that the Benjamin-Ono (BO) equation
\begin{equation}\label{BO}
\partial_t v-\hil \partial_x^2 v+v\partial_xv=0
\end{equation}
is quasilinear in the sense that none well-posedness for the IVP associated to the BO equation \eqref{BO} in $H^s(\R)$
for any $s\in\R$ can be established by an argument based only the contraction  principle argument. This important result
was proved in \cite{MST} and the argument could be used with little modifications to show that the same is true for the system in \eqref{sbo}. Thus we shall employ compactness methods in order to establish local well-posedness for the IVP \eqref{sbo}. 

\medskip

Our main result here is as follows:

\begin{thmp}\label{teo1} 
Let $s>\frac54$.  For  any $(u_{0},v_{0})\in H^{s+\frac12}(\mathbb{R})\times H^{s}(\mathbb{R})$, there exist a positive time $T=T( \|(u_{0},v_{0})\|_{H^{s+\frac{1}{2}}_{x}\times H^{s}_{x}})$, which can be chosen as a non-increasing function of its argument, and  a unique solution  $(u,v)$ of the IVP \eqref{sbo} satisfying
\begin{equation}\label{clas1}
(u,v)\in C\big([0,T]: H^{s+\frac12}_{x}(\mathbb{R})\times H^{s}_{x}(\mathbb{R})\big) 
\end{equation}
and
\begin{equation}\label{clas2}
\partial_{x}v \in L^1\big( (0,T): L^{\infty}_{x}(\mathbb{R})\big).
\end{equation}
		
Moreover, for any $0<T'<T$, there exists a neighborhood $\mathcal{U}$ of $(u_0,v_0)$ in $H^{s+\frac12}(\mathbb R) \times H^s(\mathbb R)$ such that the flow map data-to-solution 
\begin{equation*}
S: \mathcal{U} \rightarrow C\big([0,T]: H^{s+\frac12}_{x}(\mathbb{R})\times H^{s}_{x}(\mathbb{R})\big) \, , (\tilde{u}_0,\tilde{v}_0) \mapsto (\tilde{u},\tilde{v})
\end{equation*}
is continous. 
\end{thmp}

Our strategy is to use a refined Strichartz estimates as it was done by Koch and Tzvetkov in \cite{kt}, and by
Kenig and Koenig in \cite{kk} for the Benjamin-Ono equation.
There are some difficulties to overcome in order to implement this method. The first one is related to the loss
of derivatives present in the deduction of the energy estimates in $H^{s+\frac12}(\mathbb{R})\times H^{s}(\R)$ for solutions of the IVP \eqref{sbo}. To get around this obstruction, we modify the energy following the idea of Kwon in \cite{k} for the fifth-order KdV equation.   Adding an extra lower-order term to the energy permits to close the energy estimates (see also \cite{KaPi}). Note that in this case, the modified term added to the energy allows to cancel out two different bad nonlinear interactions in the energy estimates. 

Since we need the coercivity of the modified energy, this method only works in the case of small initial data. Well-posedness for arbitrary large data is generally obtained by a scaling argument, but the system \eqref{sbo} does not enjoy this special property. Nevertheless, we still can perform the following change of variables 
\begin{equation} \label{sbo_rescaled}
\begin{split}
u_{\lambda}(x,t)&=\lambda u(\lambda x, \lambda^2 t)\\
v_{\lambda}(x,t)&=\lambda v(\lambda x, \lambda^2t)
\end{split}
\end{equation}
for $0<\lambda\le 1$. Then $(u_{\lambda}, v_{\lambda})$ is solution of 
\begin{equation}\label{sbo-scale}
\begin{cases}
\mathrm{i}\partial_tu_{\la}+\partial_x^2u_{\la}=  \lambda u_{\la}v_{\la}+\beta|u_{\la}|^2u_{\la},\\
 \partial_tv_{\la}-\hil\partial_x^2 v_{\la}+ \rho v_{\la}\partial_xv_{\la}= \partial_x(|(u_{\la})|^2),\\
u_{\la}(x,0)=\la u_0(\la x), \hskip10pt v_{\la}(x,0)=\lambda v_0(\lambda x).
\end{cases}
\end{equation}
Since
\begin{equation}\label{sbo-scal1}
\begin{split}
\|u_{\la}(\cdot,0)\|_{H^{s+\frac12}} & \lesssim \la^{\frac12} (1+\la^{s+\frac12})\|u_0\|_{H^{s+\frac12}}\\
\|v_{\la}(\cdot,0)\|_{H^s} & \lesssim \la^{\frac12} (1+\la^{s})\|v_0\|_{H^{s}} ,
\end{split}
\end{equation}
given $\delta>0$, we can always choose $\la$ small enough such that 
\begin{equation}\label{sbo-scal2}
	\|(u_{\la}(\cdot,0), w_{\la}(\cdot,0))\|_{H^{s+\frac12}\times H^s}\le \delta .
\end{equation}
Moreover, suppose that there exists $\delta>0$ such that for all $0<\la\le1$, there exists a solution of the IVP \eqref{sbo-scale} $(u_{\la}, v_{\la})\in C([0,T_{\lambda}]: H^{s+\frac12}(\R)\times H^s(\R))$ whenever 
$\|(u_{\la}(\cdot,0), v_{\la}(\cdot,0))\|_{H^{s+\frac12}\times H^s}\le \delta$. Then we obtain, letting 
$(u(x,t), v(x,t))= (\la^{-1}u_{\la}(\la^{-1}x, \la^{-2}t), \la^{-1}v_{\la}(\la^{-1}x, \la^{-2}t))$, a solution of \eqref{sbo} in the function space $C([0,T]: H^{s+\frac12}(\R)\times H^s(\R))$ with a time of existence satisfying $T\gtrsim \lambda^2T_{\lambda}$. Note that the requirement on the smallness of the solutions appearing in the energy estimates must be independent of $\lambda \in (0,1]$ (see Remarks \ref{rem:en_est} and \ref{rem:en_est_diff} below). This idea was already used by Zaiter for the Ostrovsky equation in \cite{z}.

\medskip

\medskip

The plan of this paper is as follows: in Section 2, we recall some commutators estimates, which will be used in Section 3 to derive the energy estimates. Finally, Section 4 is devoted to the proof of the main theorem. 

\subsection*{Notation}
	\begin{itemize}
	\item 	For two quantities $A$ and $B$, we denote $A\lesssim B$ if $A\leq cB$ for some constant $c>0.$ Similarly, $A\gtrsim B$  if $A\geq cB$ for some $c>0.$  Also for two positive quantities, $A$ $B$ we say that are \emph{ comparable} if $A\lesssim B$ and $B\lesssim A,$ when such conditions are   satisfied we  indicate it by  writing $A\sim B.$  The dependence of the constant $c$ on other parameters or constants are usually clear from the context and we will often suppress this dependence whenever it be  possible.
		
	\item	For any pair of quantities  $X$ and $Y,$ we denote  $X\ll Y$ if $X\leq cY$ for some sufficiently small  positive constant $c.$ The smallness of such constant  is usually clear from the context.  The notation $X\gg Y$ is similarly defined.
	
	\item For $s \in \mathbb R$, $D^s_x$ and $J_x^s$ denote respectively the Riesz and Bessel potentials of order $-s$, respectively defined as Fourier multipliers by 
	\[ \mathcal{F}(D^s_xf)(\xi)=|\xi|^{s} \mathcal{F}(f)(\xi) \quad \text{and} \quad  \mathcal{F}(J^s_xf)(\xi)=(1+\xi^2)^{\frac{s}2} \mathcal{F}(f)(\xi). \]
	
	\item For $1 \le p \le \infty$, $L^p(\mathbb R)$ denotes the classical Lebesgue spaces. 
	
	\item For $s \in \mathbb R$, $H^s(\mathbb R)$ denotes the $L^2$-based sobolev space of order $s$, which consists of all distributions $f \in \mathcal{S}'(\mathbb R)$ such that $\|f\|_{H^s(\mathbb R)}=\|J_x^sf\|_{L^2} < 
	\infty$.
\end{itemize}	
\section{Commutator estimates}

The following fractional Leibniz rule was proved by Kenig, Ponce and Vega in the appendix of \cite{KPV-CPAM-93}
\begin{lem}\label{comm}
 Let $\alpha=\alpha_{1}+\alpha_{2}\in(0,1)$ with $\alpha_{1},\alpha_{2}\in (0,\alpha),p\in [1,\infty),$ and  $p_{1},p_{2}\in (1,\infty)$ such that  $\frac{1}{p}=\frac{1}{p_{1}}+\frac{1}{p_{2}}.$ Then
		\begin{equation*}
			\left\|D_x^{\alpha}(fg)-fD_x^{\alpha}g-gD_x^{\alpha}f\right\|_{L^{p}}\lesssim \|D_x^{\alpha_{1}}f\|_{L^{p_{1}}}\|D_x^{\alpha_{2}}g\|_{L^{p_{2}}}.
		\end{equation*}
		Moreover, if $p>1,$ then the case  $\alpha_{2}=0,$ the exponent $p_{2}$  with  $ 1<p_{2}\leq \infty$ is also allowed.
\end{lem}

These results as well as the Kato-Ponce commutator estimates (see \cite{kato-ponce}) were recently extended by Li  in \cite{Dli}. In particular, we will use the following estimates (see Theorem 5.1 and Corollary 5.3 in \cite{Dli}). 

\begin{lem}\label{kpcomm}
	Let   $1<p<\infty$. Let  $1<p_{1},p_{2},p_{3},p_{4}\leq \infty,$  satisfy 
	 \begin{equation*}
		\frac{1}{p_{1}}+\frac{1}{p_{2}}=\frac{1}{p_{3}}+\frac{1}{p_{4}}=\frac{1}{p}.
	\end{equation*}
	Then, for all  $f,g\in \mathcal{S}(\mathbb{R}),$ the following estimates are true. 
	\begin{enumerate}
		\item[\rm(I)] If $0<s\leq 1,$ then 
				\begin{equation*}
				\left\|\left[D_x^{s}; f\right]g\right\|_{L^{p}}\lesssim \|D_x^{s-1}\partial_{x}f\|_{L^{p_{1}}}\|g\|_{L^{p_{2}}}.
			\end{equation*}
	\item[\rm(II)] If $s>1,$ then 
	\begin{equation}\label{kp2}
		\left\|\left[D^{s}_x; f\right
		]g\right\|_{L^{p}}\lesssim \|D_x^{s}f\|_{L^{p_{1}}}\|g\|_{L^{p_{2}}_{x}}+\|\partial_{x}f\|_{L^{p_{3}}}\|D_{x}^{s-1}g\|_{L^{p_{4}}}.
	\end{equation}
	\end{enumerate}
\end{lem}
\begin{rem}
	The results in Lemmas \ref{comm} and \ref{kpcomm} are still valid  for  $\mathcal{H}D_{x}^s$ instead of $D_{x}^s$.
\end{rem}



%

In the appendix of \cite{DMP} Dawson, McGahagan and Ponce proved the following estimates.
\begin{lem}\label{propdmp}
Let $1<p,q<\infty.$ Then for all $f,g\in \mathcal{S}(\mathbb{R}),$ the following estimates are true. 
\begin{itemize}
\item[(i)]  if $l,m$ are integers such that  $l,m\geq 0$, then 
\begin{equation}\label{calderon}
		\left\|\partial_{x}^{l}\left[\mathcal{H};f\right]\partial_{x}^{m}g\right\|_{L^{p}}\lesssim_{l,m,p}\|\partial_{x}^{l+m}f\|_{L^{\infty}}\|g\|_{L^{p}}.
	\end{equation}
\item[(ii)]
if  $0\leq \alpha<1, 0<\beta\leq 1, 0<\alpha+\beta \leq 1$  and $\delta>1/q,$ then 
\begin{equation} \label{gen:calderon}
\|D_{x}^{\alpha+\beta}(gD_{x}^{1-(\alpha+\beta)}f)-D_{x}^{\alpha}(gD_{x}^{1-\alpha}f)\|_{L^{p}}
\lesssim _{\alpha,\beta,p,\delta}\|J^{\delta}\partial_{x}g\|_{L^{q}}\|f\|_{L^{p}}.
\end{equation}
\end{itemize}
\end{lem}

\begin{rem}\label{dercom}
Recently, in Proposition 3.10 of \cite{Dli}, D. Li improved estimate \eqref{gen:calderon}  by giving the following sharp version. For any $0\leq \alpha<1, 0<\beta\leq 1-\alpha$ and $1<p<\infty,$ we have 
\begin{equation*}
\|D_{x}^{\alpha+\beta}(gD_{x}^{1-(\alpha+\beta)}f)-D_{x}^{\alpha}(gD_{x}^{1-\alpha}f)\|_{L^{p}}\lesssim _{\alpha,\beta,p}\|D_{x}g\|_{\mathrm{BMO}}\|f\|_{L^{p}}.
\end{equation*}
Consequently,
\begin{equation}  \label{sharp_gen:calderon}
\|D_{x}^{\alpha+\beta}(gD_{x}^{1-(\alpha+\beta)}f)-D_{x}^{\alpha}(gD_{x}^{1-\alpha}f)\|_{L^{p}}\lesssim _{\alpha,\beta,p}\|\partial_{x}g\|_{L^{\infty}}\|f\|_{L^{p}}.
\end{equation}
\end{rem}

Finally, we list some estimates involving the Bessel and Riesz potentials. 
\begin{lem} \label{Bessel:Riesz}
The linear operators $T_1(D):=J^{-1}_x\mathcal{H}\partial_x^2-\partial_x$ and  $T_2(D)=D_x^{\frac32}J^{-1}_x-D_x^{\frac12}$ are bounded in $L^2(\mathbb R)$. More precisely, there exist $c_1>0$, $c_2>0$ such that
\begin{equation} \label{Bessel:Riesz.1}
\|T_1(D)f\|_{L^2} \leq c_1 \|f\|_{L^2}, \quad \forall \, f \in L^2(\mathbb R)
\end{equation}
and
\begin{equation} \label{Bessel:Riesz.2}
\|T_2(D)f\|_{L^2} \leq c_2 \|f\|_{L^2}, \quad \forall \, f \in L^2(\mathbb R) .
\end{equation}
\end{lem}

\begin{proof} 
Observe that $T_1(D)$ and $T_2(D)$ are Fourier multipliers whose symbols given by
\begin{eqnarray*}
T_1(\xi)&=&\frac{i\xi}{(1+|\xi|^2)^{\frac12}}\left(|\xi|-(1+|\xi|^2)^{\frac12}\right) \\ 
T_2(\xi)&=&\frac{|\xi|^{\frac12}}{(1+|\xi|^2)^{\frac12}}\left(|\xi|-(1+|\xi|^2)^{\frac12}\right)
\end{eqnarray*}
are bounded over $\mathbb R$. Estimates \eqref{Bessel:Riesz.1} and \eqref{Bessel:Riesz.2} follow then from Plancherel's identity by choosing $c_1:=\sup_{\mathbb R}|T_1(\xi)|$ and $c_2:=\sup_{\mathbb R}|T_2(\xi)|$.
\end{proof}
\section{Energy estimates} \label{Sec:EnEs}
In order to simplify the notations we will use $\rho=\beta=1$ in \eqref{sbo}. We notice that in the case that $\rho$ would be negative we has to consider in the second equation in \eqref{sbo} the operator $\hil\partial_x^2$ with a positive 
sign in front due to the trick introduced to handle the loss of derivatives in the modified energy estimate.

\subsection{Energy estimates for the solutions of \eqref{sbo}}\label{sec3.1}
As we commented in the introduction we cannot obtain an a priori estimate directly for solutions of the system \eqref{sbo}. We have extra terms that we cannot handle. More precisely, 

\begin{equation*}
2\,  \mathrm{Im}\int_{\mathbb{R}}u \, D_{x}^{s+1}\overline{u}D_{x}^{s}v\, dx
\quad \text{and} \quad
\int_{\mathbb{R}}\, D_{x}^{s}v\partial_xD_{x}^{s}(|u|^2)\, dx.
\end{equation*}

To get a priori estimates we need to modify the usual energy.  We define this new functional as follows:

\begin{defn} For $t\ge 0$ and $s\ge \frac12$,  we define the \emph{modified energy} as being
\begin{equation}\label{m-energy}
\begin{split}
E_m^s(t):=& \|u(t)\|_{L^2_{x}}^2+\|D^{s+\frac12}_{x}u(t)\|^2_{L^2_{x}} +\|v(t)\|_{L^2_{x}}^2+\frac12\|D^{s}v(t)\|^2_{L^2_{x}}\\
&\hskip10pt+\int_{\mathbb{R}} D^{s-\frac12}_xv(t)D^{s-\frac12}_x\left(|u(t)|^2\right)\,dx. 
\end{split}
\end{equation}
\end{defn}
With this definition at hand, we can establish the a priori estimates we need in this work.

\begin{prop}\label{import} Let $s>\frac12$ and $T>0$. There exist positive constants $c_s$ and $\kappa_{1,s}$ such that for any $(u,v) \in C([0,T] : H^{s+\frac12}(\mathbb R)) \times H^s(\mathbb R))$ solution of \eqref{sbo} satisfying 
\[\|\left(u(t),v(t)\right)\|_{H^{s+\frac{1}{2}}_{x}\times H^{s}_{x}}\le \frac1{c_s} ,\] the following estimates hold true. 
\begin{enumerate}
\item{} \underline{Coercivity}:
\begin{equation}\label{m-energy1}
\frac12 \Big(\|u(t)\|_{H^{s+\frac12}_{x}}^2+\|v(t)\|_{H^s_{x}}^2\Big)\le E_m^s(t)\le 
\frac32 \Big(\|u(t)\|_{H^{s+\frac12}_{x}}^2+\|v(t)\|_{H^s_{x}}^2\Big)
\end{equation}
\item{} \underline{Energy estimate}:
\begin{equation}\label{m-energy2}
\frac{d}{dt} E_m^s(t) \lesssim \Big(1+\|\partial_x v(t)\|_{L^{\infty}_{x}}\Big) E_m^s(t) ,
\end{equation}
 and as a consequence,
\begin{equation}\label{m-energy3}
\begin{split}
\sup_{t\in [0,T]} \|(u(t),v(t))\|_{H^{s+\frac12}_{x}\times H^s_{x}}
\le 2e^{\kappa_{1,s}(T+\|\partial_xv\|_{L^{1}_{T}L^{\infty}_{x}})}\|(u_0,v_0)\|_{H^{s+\frac12}_{x}\times H^s_{x}}
 .
 \end{split}
 \end{equation}
 \end{enumerate}
 \end{prop}

\begin{rem} \label{rem:en_est}
Proposition \ref{import} also holds for solutions of the system \eqref{sbo-scale} with implicit constant independent of $0<\lambda \le 1$.
\end{rem}

\begin{proof} 

%

The proof of \eqref{m-energy1} follows from the Cauchy-Schwarz inequality, the Leibniz rule and the Sobolev embedding. 
\smallskip

Let us consider $E_m^s(t)$ in \eqref{m-energy}. Differentiating with respect to $t$ it follows that
\begin{equation}\label{energy}
	\begin{split}
	\frac{d}{dt}E_{m}^s&=\frac{1}{2}\frac{d}{dt}\left(\int_{\mathbb{R}}(D_{x}^{s}v)^{2}\,d x\right)+\frac{d}{dt}\left(\int_{\mathbb{R}}|D_{x}^{s+\frac{1}{2}}u|^{2}\,dx \right)\\
	&\quad +\frac{d}{dt}\left(\int_{\mathbb{R}}D_{x}^{s-\frac12}vD_{x}^{s-\frac12}(|u|^{2})\, dx\right)+\frac{d}{dt} \left(\int_{\mathbb{R}} v^2\,dx\right)\\
	&={\rm I+II+III+IV}.
	\end{split}
\end{equation}

On one hand, using the second equation in \eqref{sbo} we have,
\begin{equation}\label{energy-1}
	\begin{split}
	{\rm I}&=-\int_{\mathbb{R}}D_{x}^{s}vD_{x}^{s}(v\partial_{x}v)\, dx+\int_{\mathbb{R}}D_{x}^{s}v\,\partial_{x}D_{x}^{s}(|u|^{2})\, dx\\
	&=I_{1}+I_{2}.
	\end{split}
\end{equation}
We obtain after integrating by parts 
\begin{equation*}
	\begin{split}
		{\rm I}_{1}&=-\int_{\mathbb{R}}D_{x}^{s}v\, [D_{x}^{s}; v]\partial_{x}v\, dx+\frac{1}{2}\int_{\mathbb{R}}\partial_{x}v\, (D_{x}^{s}v)^{2}\, dx\\
		&={\rm I}_{1,1}+{\rm I}_{1,2}.
	\end{split}
\end{equation*}
Thus, in virtue of  Lemma \ref{kpcomm} we  get 
\begin{equation*}
	{\rm |I}_{1,1}|\lesssim \|\partial_{x}v\|_{L^{\infty}}\|D_{x}^{s}v\|_{L^{2}}^{2}.
\end{equation*}
For ${\rm I}_{1,2}$, it readily follows that 
\begin{equation*}
	{\rm |I}_{1,2}|\lesssim \|\partial_{x}v\|_{L^{\infty}}\|D_{x}^{s}v\|_{L^{2}}^{2}.
\end{equation*}
Thus
\begin{equation}\label{term1-1}
{\rm |I}_{1}| \lesssim \|\partial_{x}v\|_{L^{\infty}}\|D_{x}^{s}v\|_{L^{2}}^{2}.
\end{equation}

Later on we will come back  on ${\rm I}_{2}$ since it requires to obtain the full description of  ${\rm III}$.
\medskip

After using integration by parts and system \eqref{sbo} we have that
\begin{equation*}
\begin{split}
{\rm II}&=-2\, \mathrm{Re}\left(\mathrm{i}\int_{\mathbb{R}}D_{x}^{s+\frac{1}{2}}\overline{u}D_{x}^{s+\frac{1}{2}}(uv)\, dx\right)\\
		&\quad -2 \,\mathrm{Re}\left(\mathrm{i}\int_{\mathbb{R}}D_{x}^{s+\frac{1}{2}}\overline{u} D_{x}^{s+\frac{1}{2}}(u|u|^{2})\, dx\right)\\
		&={\rm II}_{1}+{\rm II}_{2}.
	\end{split}
\end{equation*}
In the case of ${\rm II}_{1}$  we rewrite it as 
\begin{equation*}
	\begin{split}
		{\rm II}_{1}&=2\,  \mathrm{Im}\int_{\mathbb{R}}D_{x}^{s+1}\overline{u}[D_{x}^{s}; u]v\, dx +2\,  \mathrm{Im}\int_{\mathbb{R}}u \, D_{x}^{s+1}\overline{u}D_{x}^{s}v\, dx\\
		&={\rm II}_{1,1}+{\rm II}_{1,2}.
	\end{split}
\end{equation*}

Notice that  for ${\rm II}_{1,1}$ we have
\begin{equation*}
\begin{split}
	{\rm II}_{1,1}&=2 \,\mathrm{Im}\int_{\mathbb{R}}D_{x}^{s+\frac{1}{2}}\overline{u}\, D_{x}^{\frac{1}{2}}\left[D_{x}^{s}; u \right]v\, dx\\
	&=2\,\mathrm{Im}\int_{\mathbb{R}}D_{x}^{s+\frac{1}{2}}\overline{u}\,[D_{x}^{s+\frac12}; u]v\, dx-2\,\mathrm{Im}\int_{\mathbb{R}}D_{x}^{s+\frac{1}{2}}\overline{u}\,[D_{x}^{\frac{1}{2}}; u]D_{x}^{s}v\, dx\\
	&={\rm II}_{1,1,1}+{\rm II}_{1,1,2}.
\end{split}
\end{equation*}

The terms ${\rm II}_{1,1,1}$ and ${\rm II}_{1,1,2}$ can be controlled after applying Lemma \ref{kpcomm} and the Sobolev embedding (recalling $s>\frac12$).   In the first place, we have
\begin{equation}\label{term2-1}
\begin{split}
|{\rm II}_{1,1,1}|&\lesssim \|D_{x}^{s+\frac{1}{2}}u\|_{L^{2}} \Big(\|D_{x}^{s+\frac{1}{2}}u\|_{L^{2}}\|v\|_{L^{\infty}}+\|\partial_xu\|_{L^{2_+}}\|D^{s-\frac12}v\|_{L^{\infty_-}}\Big)\\
&\lesssim 
 \|u\|_{H^{s+\frac12}}^2\|v\|_{H^s},
\end{split}
\end{equation}
and in the second place, we get
\begin{equation}\label{term3-1}
|{\rm II}_{1,1,2}| 
\lesssim \|D^{s+\frac12}_xu\|_{L^2}\|D^s_xv\|_{L^2}\|D^{-\frac12}_x\partial_xu\|_{L^{\infty}} \lesssim \|u\|_{H^{s+\frac12}}^2\|v\|_{H^s}.
\end{equation}
To handle ${\rm II}_{1,2}$  we  require to deploy the terms that compounds ${\rm III},$  so that for the moment  we will  continue estimating the  remainder terms. 

Concerning ${\rm II}_2$  we have since $H^{s+\frac12}$ is a Banach algebra that 
\begin{equation}\label{term4-1}
|{\rm II}_{2}|\lesssim \|u\|_{H^{s+\frac12}}^{4} .
\end{equation}
Next we  focus our attention on ${\rm III}$ that is by itself the most complicated term, since it contains the interaction between $u$ and $v$. In the first place,
\begin{equation*}
	\begin{split}
		{\rm III}&= \int_{\mathbb{R}}D_{x}^{s-\frac12}v_{t}D_{x}^{s-\frac12}(|u|^{2})\, dx
		+ 2\, \mathrm{Re}\left(\int_{\mathbb{R}}D_{x}^{s-\frac12}vD_{x}^{s-\frac12}(\overline{u}u_{t})\, dx\right)\\
		&={\rm III}_{1}+{\rm III}_{2}.
	\end{split}
\end{equation*}
Since $v$ satisfies the BO equation in system \eqref{sbo} is quite clear that 
\begin{equation*}
	\begin{split}
		{\rm III}_{1}&= \int_{\mathbb{R}}D_{x}^{s-1}\mathcal{H}\partial_{x}^{2}v \, D_{x}^{s}(|u|^{2})\, dx
		-\frac12 \int_{\mathbb{R}}D_{x}^{s-1}\partial_x(v^2) D_{x}^{s}(|u|^{2})\, dx\\
		&={\rm III}_{1,1}+{\rm III}_{1,2}.
	\end{split}
\end{equation*}
After observing that $D^1_x=\mathcal{H}\partial_x$, we obtain that
\begin{equation*}
{\rm III}_{1,1}+{\rm I}_{2}=0 .
\end{equation*}
In the case of ${\rm III}_{1,2}$, we deduce from the Cauchy-Schwarz inequality and the fact that $H^s$ is a Banach algebra for $s>\frac{1}{2}$, that
\begin{equation}\label{term5-1}
{\rm |III}_{1,2}| \lesssim \|D_{x}^{s}(v^2)\|_{L^{2}}\|D_{x}^{s}(|u|^2)\|_{L^{2}} \lesssim  \|u\|_{H^{s}}^{2}\|v\|_{H^{s}}^{2} .
\end{equation}
Now we turn our attention to ${\rm III}_{2}$. We obtain after using the Schr\"{o}dinger equation in \eqref{sbo} and integrating by parts that
\begin{equation*}
{\rm III}_{2}=2 \, \mathrm{Re}\left(\mathrm{i}\int_{\mathbb{R}}D_{x}^{s-\frac12}vD_{x}^{s-\frac12}(\overline{u}\partial_{x}^{2}u)\, dx\right)
=2\, \mathrm{Im} \int_{\mathbb{R}}D_{x}^{s-1}\partial_{x}v D_{x}^{s}(\overline{u}\partial_{x}u)\, dx .
\end{equation*}
Since $\partial_{x}=-\mathcal{H}D_{x}$  it follows  that
\begin{equation*}
	\begin{split}
		{\rm III}_2&=2\,\mathrm{Im}\int_{\mathbb{R}}D_{x}^{s}v\, [\mathcal{H}D_{x}^{s}; \overline{u}]\partial_{x}u \, dx+2\,\mathrm{Im}\int_{\mathbb{R}}\overline{u} \, D_{x}^{s}vD_{x}^{s+1}u\, dx\\
		&={\rm III}_{2,1}+{\rm III}_{2,2}.
	\end{split}
\end{equation*}
After gathering together ${\rm III}_{2,2}$ and ${\rm II}_{1,2}$   we get 
\begin{equation*}
	{\rm II}_{1,2}+{\rm III}_{2,2}=0 .
\end{equation*}
For ${\rm III}_{2,1}$ there is no cancellation. Instead  we estimate directly by using Lemma \ref{kpcomm} 
\begin{equation*}
|{\rm III}_{2,1}|\lesssim \|D_{x}^{s}v(t)\|_{L^{2}}\big\|[\mathcal{H}D^s_x, \overline{u}]\partial_xu \big\|_{L^2}
\lesssim \|D_{x}^{s}v(t)\|_{L^{2}}\|D^{s-1}\partial_xu\|_{L^{\infty_-}}\|\partial_xu\|_{L^{2_+}}
\end{equation*}
so that
\begin{equation}\label{term6-1}
|{\rm III}_{2,1}|\lesssim  \|D_{x}^{s}v\|_{L^{2}}\|u\|_{H^{s+\frac12}}^2
\end{equation}
since $s>\frac12$. 

Finally, we estimate ${\rm IV}$. Using the second equation in \eqref{sbo}, integrating by parts and using properties of the Hilbert transform we have
\begin{equation}
\frac{d}{dt} \int_{\mathbb{R}} v^2\,dx= 2\int_{\mathbb{R}} v v_t\,dx\\
=2\int v\partial_x(|u|^2)\,dx.
\end{equation} 
Thus it follows since $H^{s+\frac12}$ is a Banach algebra for $s>\frac12$ that 
\begin{equation}\label{term7-1}
|{\rm IV}| \lesssim  \|v\|_{L^{2}}\|u^2\|_{H^{1}} \lesssim  \|v\|_{L^{2}}\|u\|_{H^{s+\frac12}}^2 .
\end{equation}

Gathering the estimates from \eqref{energy} to \eqref{term7-1}, using the definition of $E_m^s(t)$ and requiring  $\|\left(u(t),v(t)\right)\|_{H^{s+\frac{1}{2}}\times H^{s}}\le \frac1{c_s}$  we obtain
\begin{equation} \label{est:energy.100}
\frac{d}{dt} E_m^s(t)\lesssim \Big(1+\|\partial_x v(t)\|_{L^{\infty}}\Big) E_m^s(t),
\end{equation}

The proof of estimate \eqref{m-energy2} follows by by combining \eqref{m-energy1} and \eqref{est:energy.100}, while  estimate \eqref{m-energy3} is deduced by applying Gronwall's inequality.

\end{proof}

\subsection{Energy estimates for the differences of two solutions of \eqref{sbo}}\label{sec3.2}

Let $(u_1,v_1)$ and $(u_2,v_2)$ be two solutions of \eqref{sbo}. We define 
\begin{equation} \label{def:wz}
\begin{cases} 
w=u_1-u_2, & u=u_1+u_2  ,\\ 
z=v_1-v_2, & v=v_1+v_2 .
\end{cases}
\end{equation}
Then $(w,z)$ is solution of the system 
\begin{equation}\label{sbo:diff}
\begin{cases}
\mathrm{i}\partial_t w+\partial_x^2 w = \frac12 (vw+uz)+\frac14|w|^2w+\frac12|u|^2w+\frac14 u^2\overline{w},\\
\;\partial_tz-\hil \partial_x^2 z+ \frac12\partial_x(vz)=\partial_x\mathrm{Re}(\overline{u}w).
\end{cases}
\end{equation}

The aim of this subsection is to derive energy estimates on the difference of two solutions $(w,z)$. As in subsection \ref{sec3.1}, we define a modified energy for $(w,z)$. 

\begin{defn} Let $t\ge 0$. 

(i) \underline{Case $\sigma=0$}.
\begin{equation}\label{m-energy:diff:0}
\widetilde{E}_m^0(t):= \|w(t)\|_{L^2_{x}}^2+\|D^{\frac12}_{x}w(t)\|^2_{L^2_{x}} +\frac12\|z(t)\|_{L^2_{x}}^2+\mathrm{Re} \int_{\mathbb{R}} J_x^{-1}z\left(\overline{u}w\right)(t)\,dx. 
\end{equation}
where $J_x^{-1}$ is the Bessel potential of order $1$, defined in the notation.

(ii) \underline{Case $\sigma=s \ge \frac12$}.
\begin{equation}\label{m-energy:diff:s}
\begin{split}
\widetilde{E}_m^s(t):=& \|w(t)\|_{L^2_{x}}^2+\|D^{s+\frac12}_{x}w(t)\|^2_{L^2_{x}} +\|z(t)\|_{L^2_{x}}^2+\frac12\|D^{s}z(t)\|^2_{L^2_{x}}\\
&\hskip10pt+\mathrm{Re} \int_{\mathbb{R}} D^{s-\frac12}_xz(t)D^{s-\frac12}_x\left(\overline{u}w\right)\,dx. 
\end{split}
\end{equation}

\end{defn}
With this definition at hand, we can establish the a priori estimates we need in this work.

\begin{prop}\label{en_est:diff} Let $s>\frac12$ and $T>0$. There exists a positive constant $\widetilde{c}_s$ such that for any $(u_1,v_1), \, (u_2,v_2) \in C([0,T] : H^{s+\frac12}(\mathbb R)) \times H^s(\mathbb R))$ solutions of \eqref{sbo} satisfying 
\begin{equation} \label{m-energy:diff.0}
\|\left(u_i(t),v_i(t)\right)\|_{H^{s+\frac{1}{2}}_{x}\times H^{s}_{x}}\le \frac1{\widetilde{c}_s} , \quad i=1,2,
\end{equation}
the following estimates hold true. 
\begin{enumerate}
\item{} \underline{Coercivity}: for $\sigma=0$ or $\sigma=s>\frac12$
\begin{equation}\label{m-energy:diff.1}
\frac12 \Big(\|w(t)\|_{H^{\sigma+\frac12}_{x}}^2+\|z(t)\|_{H^{\sigma}_{x}}^2\Big)\le \widetilde{E}_m^{\sigma}(t)\le 
\frac32 \Big(\|w(t)\|_{H^{\sigma+\frac12}_{x}}^2+\|z(t)\|_{H^{\sigma}_{x}}^2\Big).
\end{equation}
\item{} \underline{$H^\frac12 \times L^2$-energy estimate}:
\begin{equation}\label{m-energy:diff.2}
\frac{d}{dt} \widetilde{E}_m^0(t) \lesssim \Big(1+\|\partial_x v_1(t)\|_{L^{\infty}_{x}}+\|\partial_x v_2(t)\|_{L^{\infty}_{x}}\Big) \widetilde{E}_m^0(t).
\end{equation}
\item{} \underline{$H^{s+\frac12} \times H^{s}$-energy estimate}:
\begin{equation}\label{m-energy:diff.3}
\frac{d}{dt} \widetilde{E}_m^s(t)  \lesssim \Big(1+\|\partial_x v_1(t)\|_{L^{\infty}_{x}}+\|\partial_x v_2(t)\|_{L^{\infty}_{x}}\Big) \widetilde{E}_m^s(t) +f_s(t)
\end{equation}
where $f_s=f_s(t)$ is defined by
\begin{equation} \label{m-energy:diff.3bis}
\begin{split}
f_s(t)& =\|D^s_x\partial_xv\|_{L^{\infty}}\|D^s_xz\|_{L^2}\|z\|_{L^2} +\|D^s_xv\|_{L^2}\|D^s_xz\|_{L^2}\|\partial_xz\|_{L^{\infty}} \\ 
&\quad +  \|D_x^{s+\frac12}v\|_{L^{\infty}}\|w\|_{L^2}\|D_x^{s+\frac12}w\|_{L^2}+\|D^{s+1}u\|_{L^{\infty}}\|w\|_{L^2}\|D_x^sz\|_{L^2}.
\end{split}
\end{equation}
 \end{enumerate}
 \end{prop}
 
 \begin{rem} \label{rem:en_est_diff}
Proposition \ref{en_est:diff} also holds for solutions of the system \eqref{sbo_rescaled} with implicit constant independent of $0<\lambda \le 1$.
\end{rem}

 \begin{rem} The terms gathered in $f_s(t)$ cannot be estimated directly, but they have always more derivatives on the functions $u_i$, $v_i$ than on the terms for the differences $w$ and $z$. These terms will be handled with the Bona-Smith argument as in Proposition 2.18 in \cite{lps}.    \end{rem}

\begin{proof} The proof of \eqref{m-energy:diff.1} follows from the Cauchy-Schwarz inequality, the Leibniz rule and the Sobolev embedding. 
\smallskip

Next, we  prove \eqref{m-energy:diff.2}. Differentiating with respect to $t$ it follows that
\begin{equation}\label{energy:diff.4}
	\begin{split}
	\frac{d}{dt}\widetilde{E}_{m}^0&=\frac{1}{2}\frac{d}{dt}\left(\int_{\mathbb{R}}z^{2}\,d x\right)+\frac{d}{dt}\left(\int_{\mathbb{R}}|w|^{2}\,dx \right)\\
	&\quad +\frac{d}{dt}\left(\int_{\mathbb{R}}|D^{\frac12}_xw|^{2}\,dx \right)+\frac{d}{dt}\left(\mathrm{Re} \int_{\mathbb{R}}J^{-1}_xz\left(\overline{u}w\right)\, dx\right)\\
	&={\rm \widetilde{I}+\widetilde{II}+\widetilde{III}+\widetilde{IV}}.
	\end{split}
\end{equation}

We compute each term separately. After using the second equation in \eqref{sbo:diff} and integrating by parts, we find that 
\begin{displaymath}
{\rm \widetilde{I}}=\frac14 \int_{\mathbb R} (\partial_xv) z^2 dx +\mathrm{Re} \int_{\mathbb R} z \partial_x(\overline{u}w) dx={\rm \widetilde{I}_1}+{\rm \widetilde{I}_2 }.
\end{displaymath}
By using H\"older's inequality  ${\rm \widetilde{I}_1}$ is estimated easily by 
\begin{equation} \label{energy:diff.5}
|{\rm \widetilde{I}_1}| \lesssim \|\partial_xv\|_{L^{\infty}}\|z\|_{L^2}^2 .
\end{equation}
On the other hand, ${\rm \widetilde{I}_2}$ cannot be estimated directly and will instead be cancelled out by a term coming from ${\rm \widetilde{IV}}$. 

To deal with ${\rm \widetilde{II}}$, we use the first equation in \eqref{sbo:diff}, H\"older's inequality and the Soboelev embedding to deduce that 
\begin{equation} \label{energy:diff.6}
\begin{split}
|{\rm \widetilde{II}}|& \le  \left|\mathrm{Im} \int_{\mathbb R} \overline{w} u z dx \right| +\frac12 \left|\mathrm{Im} \int_{\mathbb R} u^2\overline{w}^2 dx \right| \\
&  \lesssim \|u\|_{H^s}\|w\|_{L^2}\|z\|_{L^2}+\|u\|_{H^s}^2\|w\|_{L^2}^2 .
\end{split}
\end{equation}

Next, we decompose ${\rm \widetilde{III}}$ as
\begin{displaymath}
\begin{split}
{\rm \widetilde{III}}&=\mathrm{Im} \int_{\mathbb R} D^{\frac12}_x\overline{w}D^{\frac12}_x(vw)dx+\mathrm{Im} \int_{\mathbb R} D^{\frac12}_x\overline{w}D^{\frac12}_x(uz)dx  +\frac12\mathrm{Im} \int_{\mathbb R} D^{\frac12}_x\overline{w}D^{\frac12}_x(|w|^2w)dx
\\ &\quad+
\mathrm{Im} \int_{\mathbb R} D^{\frac12}_x\overline{w}D^{\frac12}_x(|u|^2w)dx
+\frac12\mathrm{Im} \int_{\mathbb R} D^{\frac12}_x\overline{w}D^{\frac12}(u^2\overline{w})dx \\ 
&={\rm \widetilde{III}_1}+{\rm \widetilde{III}_2}+{\rm \widetilde{III}_3}+{\rm \widetilde{III}_4}+{\rm \widetilde{III}_5}.
\end{split}
\end{displaymath}
First it follows from the Cauchy-Schwarz inequality, the Kato-Ponce estimate in Lemma \ref{kpcomm}, and the Sobolev embedding (recalling $s>\frac12$) that 
\begin{equation} \label{energy:diff.7}
\begin{split}
| {\rm \widetilde{III}_1} |&=\left| \mathrm{Im} \int_{\mathbb R} D^{\frac12}_x\overline{w}[D^{\frac12}_x, v]w dx \right| \\ & \le \|D^{\frac12}_x w\|_{L^2}\|D^{\frac12}_x v\|_{L^{2_+}}\|w\|_{L^{\infty_-}} \lesssim \|v\|_{H^s}\|w\|_{H^{\frac12}}^2 .
\end{split}
\end{equation}
Secondly, we write ${\rm \widetilde{III}_2}$ as
\begin{equation*}
 {\rm \widetilde{III}_2} = \mathrm{Im} \int_{\mathbb R} (D^{\frac12}_x\overline{w})uD^{\frac12}_xzdx+
 \mathrm{Im} \int_{\mathbb R} D^{\frac12}_x\overline{w}[D^{\frac12}_x,u]zdx=  {\rm \widetilde{III}_{2,1}}+ {\rm \widetilde{III}_{2,2}}.
\end{equation*}
While ${\rm \widetilde{III}_{2,1}}$ cannot be handled directly and will be cancelled out by a term coming from ${\rm \widetilde{IV}}$, we use the Cauchy-Schwarz inequality, the Kato-Ponce estimate in Lemma \ref{kpcomm}, and the Sobolev embedding (recalling $s>\frac12$) to get that
\begin{equation} \label{energy:diff.8}
| {\rm \widetilde{III}_{2,2}} | \lesssim \|D^{\frac12}_x w\|_{L^2}\|D^{\frac12}u\|_{L^{\infty}}\|z\|_{L^2} 
\lesssim \|u\|_{H^{s+\frac12}} \|D^{\frac12}_x w\|_{L^2}\|z\|_{L^2}.
\end{equation}
Finally, the fractional Leibniz rule (see Lemma \ref{comm}) and the Sobolev embedding yield
\begin{equation} \label{energy:diff.9}
| {\rm \widetilde{III}_{3}} | \lesssim \|w\|_{L^{\infty}}^2 \|D^{\frac12}_xw\|_{L^2}^2 \lesssim \left(\|u_1\|_{H^{s+\frac12}}+\|u_2\|_{H^{s+\frac12}} \right)^2 \|w\|_{H^{\frac12}}^2,
\end{equation}
while the Kato-Ponce inequality (see Lemma \ref{kpcomm}) and the Sobolev embedding yield
\begin{equation} \label{energy:diff.10bis}
\begin{split}
| {\rm \widetilde{III}_{4}} | &= \left|\int_{\mathbb R}D^{\frac12}_x\overline{w}[D^{\frac12}_x,|u|^2]w dx \right|  
\lesssim \|D^{\frac12}(|u|^2)\|_{L^{\infty}} \|w\|_{H^{\frac12}}^2 
 \lesssim \|u\|_{H^{s+\frac12}}^2\|w\|_{H^{\frac12}}^2 .
\end{split}
\end{equation}
and
\begin{equation} \label{energy:diff.10}
\begin{split}
| {\rm \widetilde{III}_{5}} | &= \frac12\left| \int_{\mathbb R} u^2 (D^{\frac12}_x\overline{w} )^2 dx+\int_{\mathbb R}D^{\frac12}_x\overline{w}[D^{\frac12}_x,u^2]\overline{w} dx \right| \\ 
& \lesssim \left( \|u\|_{L^{\infty}}^2+\|D^{\frac12}(u^2)\|_{L^{\infty}}\right) \|w\|_{H^{\frac12}}^2 
\\ & \lesssim \|u\|_{H^{s+\frac12}}^2\|w\|_{H^{\frac12}}^2 .
\end{split}
\end{equation}

Finally, we deal with ${\rm \widetilde{IV}}$. We get after differentiating in time
\begin{equation*} 
\begin{split}
{\rm \widetilde{IV}}&=\mathrm{Re} \int_{\mathbb{R}} (J^{-1}_x\partial_tz) \overline{u}w\,dx
+\mathrm{Re} \int_{\mathbb{R}} (J^{-1}_x z) (\overline{\partial_t u})w\,dx+\mathrm{Re} \int_{\mathbb{R}} (J^{-1}_xz) \overline{u} (\partial_tw)\,dx
\\ & = {\rm \widetilde{IV}_1}+{\rm \widetilde{IV}_2}+{\rm \widetilde{IV}_3} .
\end{split}
\end{equation*}
By using the second equation in \eqref{sbo:diff}, we observe that
\begin{equation*} 
{\rm \widetilde{IV}_1}=\mathrm{Re} \int_{\mathbb{R}} (J^{-1}_x\mathcal H \partial_x^2z) \overline{u}w\,dx
-\frac12 \mathrm{Re} \int_{\mathbb{R}} J^{-1}_x \partial_x(vz) \overline{u}w\,dx={\rm \widetilde{IV}_{1,1}}+{\rm \widetilde{IV}_{1,2}},
\end{equation*}
where we used that $ \int_{\mathbb{R}} J^{-1}_x \partial_x \mathrm{Re}( \overline{u}w)  \mathrm{Re}( \overline{u}w)\,dx=0$. On the one hand, we rewrite ${\rm \widetilde{IV}_{1,1}}$ as
\begin{equation*} 
{\rm \widetilde{IV}_{1,1}}=\mathrm{Re} \int_{\mathbb{R}} \left((J^{-1}_x\mathcal H \partial_x^2-\partial_x)z\right) \overline{u}w\,dx+
\mathrm{Re} \int_{\mathbb{R}} (\partial_xz) \overline{u}w\,dx
={\rm \widetilde{IV}_{1,1,1}}+{\rm \widetilde{IV}_{1,1,2}} .
\end{equation*}
It follows from Lemma \ref{Bessel:Riesz} that 
\begin{equation}  \label{energy:diff.11}
|{\rm \widetilde{IV}_{1,1,1}}| \lesssim \|u\|_{H^s} \|z\|_{L^2}\|w\|_{L^2},
\end{equation}
and we use the identity 
\begin{equation} \label{energy:diff.12}
{\rm \widetilde{I}_{2}}+{\rm \widetilde{IV}_{1,1,2}}=0 
\end{equation}
to handle ${\rm \widetilde{IV}_{1,1,2}}$. On the other hand, we deduce from H\"older's inequality and the Sobolev embedding that 
\begin{equation}  \label{energy:diff.13}
|{\rm \widetilde{IV}_{1,2}}| \lesssim \|v\|_{L^{\infty}} \|z\|_{L^2}\|u\|_{L^{\infty}}\|w\|_{L^2} \lesssim \|v\|_{H^s} \|u\|_{H^s}\|z\|_{L^2}\|w\|_{L^2} .
\end{equation}

By using the first equation in \eqref{sbo}, we decompose ${\rm \widetilde{IV}_2}$ as
\begin{equation*}
\begin{split}
{\rm \widetilde{IV}_2}&=\mathrm{Im} \int_{\mathbb{R}} (J^{-1}_x z) (\partial_x^2\overline{ u})w\,dx
-\mathrm{Im} \int_{\mathbb{R}} (J^{-1}_x z)\left( \overline{u_1} v_1 +\overline{u_2} v_2\right)w\,dx \\ &
\quad -\mathrm{Im} \int_{\mathbb{R}} (J^{-1}_x z)\left( |u_1|^2\overline{u_1}+|u_2|^2\overline{u_2} \right)w\,dx \\ 
&={\rm \widetilde{IV}_{2,1}}+{\rm \widetilde{IV}_{2,2}}+{\rm \widetilde{IV}_{2,3}} .
\end{split}
\end{equation*}
Observe after integrating by parts that 
\begin{equation*}
\begin{split}
{\rm \widetilde{IV}_{2,1}} &=-\mathrm{Im} \int_{\mathbb{R}} (J^{-1}_x \partial_x z) (\partial_x\overline{u})w\,dx-\mathrm{Im} \int_{\mathbb{R}} (J^{-1}_x z) (\partial_x\overline{u})(\partial_x w)\,dx
\\ &  = {\rm \widetilde{IV}_{2,1,1}}+{\rm \widetilde{IV}_{2,1,2}} .
\end{split}
\end{equation*}
We deduce from H\"older's inequality and the Sobolev embedding (under the restriction $s>\frac12$) that
\begin{equation} \label{energy:diff.14}
|{\rm \widetilde{IV}_{2,1,1}}|  \lesssim \| J^{-1}_x \partial_x z \|_{L^2} \|\partial_xu\|_{L^{2_+}}\|w\|_{L^{\infty_-}} 
\lesssim \|u\|_{H^{s+\frac12}} \|z\|_{L^2}\|w\|_{H^{\frac12}} .
\end{equation}
The contribution ${\rm \widetilde{IV}_{2,1,2}}$ will be compensated by a term coming form ${\rm \widetilde{IV}_{3}}$.
Moreover, H\"older's inequality and the Sobolev embedding imply 
\begin{eqnarray} \label{energy:diff.15}
|{\rm \widetilde{IV}_{2,2}}| &\lesssim &
\left( \|u_1\|_{H^s} \|v_1\|_{L^2}+\|u_2\|_{H^s} \|v_2\|_{L^2}\right)\|z\|_{L^2}\|w\|_{L^2} ; \\
|{\rm \widetilde{IV}_{2,3}}| &  \lesssim &
\left( \|u_1\|_{H^s}+ \|u_2\|_{H^s}\right)^3\|z\|_{L^2}\|w\|_{L^2} .
\end{eqnarray}

Now, we use the first equation in \eqref{sbo:diff} to decompose ${\rm \widetilde{IV}_3}$ as 
\begin{equation*} 
\begin{split}
{\rm \widetilde{IV}_3}&=-\mathrm{Im} \int_{\mathbb{R}} (J^{-1}_xz) \overline{u} (\partial_x^2w)\,dx
+\frac12\mathrm{Im} \int_{\mathbb{R}} (J^{-1}_xz) \overline{u} (vw+uz)\,dx \\
& \quad + \mathrm{Im} \int_{\mathbb{R}} (J^{-1}_xz) \overline{u} (\frac14|w|^2w+\frac12|u|^2w+\frac14u^2\overline{w})\,dx \\ 
&= {\rm \widetilde{IV}_{3,1}}+{\rm \widetilde{IV}_{3,2}}+{\rm \widetilde{IV}_{3,3}} .
\end{split}
\end{equation*}
By integration by parts, we have
\begin{equation*} 
\begin{split}
{\rm \widetilde{IV}_{3,1}}&=\mathrm{Im} \int_{\mathbb{R}} (J^{-1}_x\partial_xz) \overline{u} (\partial_xw)\,dx
+\mathrm{Im} \int_{\mathbb{R}} (J^{-1}_xz) (\partial_x\overline{u})(\partial_xw)\,dx 
\\ & ={\rm \widetilde{IV}_{3,1,1}}+{\rm \widetilde{IV}_{3,1,2}}.
\end{split}
\end{equation*}
On the one hand, observe that 
\begin{equation} \label{energy:diff.15b}
{\rm \widetilde{IV}_{2,1,2}}+{\rm \widetilde{IV}_{3,1,2}}=0.
\end{equation}
On the other hand by using $D^1_x=\mathcal{H}\partial_x$, we decompose ${\rm \widetilde{IV}_{3,1,1}}$ as
\begin{equation*}
\begin{split}
{\rm \widetilde{IV}_{3,1,1}}
&=\mathrm{Im} \int_{\mathbb{R}} (D_x^{\frac12}z) \overline{u} (D_x^{\frac12}w)+\mathrm{Im} \int_{\mathbb{R}} (D_x^{\frac32}J^{-1}_x-D^{\frac12}_x)z \overline{u} (D_x^{\frac12}w) \\ 
&\quad +\mathrm{Im} \int_{\mathbb{R}} [\mathcal{H}D_x^{\frac12},\overline{u} ]J^{-1}_x\partial_xz (D_x^{\frac12}w)\,dx \\ 
&={\rm \widetilde{IV}_{3,1,1,1}}+{\rm \widetilde{IV}_{3,1,1,2}}+{\rm \widetilde{IV}_{3,1,1,3}} .
\end{split}
\end{equation*}
Then, we use the cancellation 
\begin{equation} \label{energy:diff.16}
{\rm \widetilde{III}_{2,1}}+{\rm \widetilde{IV}_{3,1,1,1}}=0.
\end{equation}
Moreover, estimate \eqref{Bessel:Riesz.2} implies 
\begin{equation} \label{energy:diff.17}
|{\rm \widetilde{IV}_{3,1,1,2}}| \lesssim  \|u\|_{L^{\infty}}\|z\|_{L^2}\|D^{\frac12}_xw\|_{L^2}.
\end{equation}
The Kato-Ponce commutator estimate in Lemma \ref{kpcomm} and the Sobolev embedding yield 
\begin{equation} \label{energy:diff.18}
|{\rm \widetilde{IV}_{3,1,1,3}}| \lesssim  \|D_x^{\frac12}u\|_{L^{\infty}}\|z\|_{L^2}\|D^{\frac12}_xw\|_{L^2} \lesssim \|u\|_{H^{s+\frac12}}\|z\|_{L^2}\|D^{\frac12}_xw\|_{L^2}.
\end{equation}

Finally, H\"older's inequality and the Sobolev embedding imply 
\begin{eqnarray} 
|{\rm \widetilde{IV}_{3,2}}| &\lesssim &
\|u\|_{H^s}\|z\|_{L^2}\left( \|v\|_{H^s} \|w\|_{L^2}+\|u\|_{H^s} \|z\|_{L^2}\right) ; \label{energy:diff.19}\\
|{\rm \widetilde{IV}_{3,3}}| &  \lesssim &
\left( \|u_1\|_{H^s}+ \|u_2\|_{H^s}\right)^3\|z\|_{L^2}\|w\|_{L^2} . \label{energy:diff.20}
\end{eqnarray}

Therefore, we conclude the proof of \eqref{m-energy:diff.2} gathering \eqref{energy:diff.4}--\eqref{energy:diff.20}.
\medskip

Now, we turn to the proof of \eqref{m-energy:diff.3}. Differentiating with respect to $t$ it follows that
\begin{equation}\label{energy:diff.21}
	\begin{split}
	\frac{d}{dt}\widetilde{E}_{m}^s&=\frac{d}{dt}\left(\int_{\mathbb{R}}z^{2}\,d x\right)+\frac{d}{dt}\left(\int_{\mathbb{R}}|w|^{2}\,dx \right)+\frac12\frac{d}{dt}\left(\int_{\mathbb{R}}(D^s_xz)^{2}\,d x\right)\\
	&\quad +\frac{d}{dt}\left(\int_{\mathbb{R}}|D^{s+\frac12}_xw|^{2}\,dx \right)+\frac{d}{dt}\left(\mathrm{Re} \int_{\mathbb{R}}D^{s-\frac12}_xzD^{s-\frac12}_x\left(\overline{u}w\right)\, dx\right)\\
	&= \widetilde{\mathcal{I}}+\widetilde{\mathcal{II}}+\widetilde{\mathcal{III}}+\widetilde{\mathcal{IV}}+\widetilde{\mathcal{V}}.
	\end{split}
\end{equation}

We argue as in the proof of \eqref{energy:diff.4} and compute each term on the right-hand side of \eqref{energy:diff.21}. After integrating by parts, we get that 
\begin{displaymath}
\widetilde{\mathcal{I}}=-\frac12 \int_{\mathbb R} (\partial_xv) z^2 dx +2\mathrm{Re} \int_{\mathbb R} z \partial_x(\overline{u}w) dx=\widetilde{\mathcal{I}}_1+\widetilde{\mathcal{I}}_2.
\end{displaymath}
Then, we deduce from H\"older's inequality that
\begin{equation} \label{energy:diff.22}
|\widetilde{\mathcal{I}}_1| \lesssim \|\partial_xv\|_{L^{\infty}}\|z\|_{L^2}^2  
\end{equation}
and
\begin{equation} \label{energy:diff.220}
\begin{split}
|\widetilde{\mathcal{I}}_2|  
&\lesssim \|z\|_{L^2}\|\partial_xu\|_{L^2}\|w\|_{L^{\infty}}+\|z\|_{L^2}\|u\|_{L^{\infty}}\|\partial_xw\|_{L^{2}} \\ &
\lesssim \|z\|_{L^2}\|u\|_{H^{s+\frac12}} \|w\|_{H^{s+\frac12}} ,
\end{split}
\end{equation}
where we used also the Sobolev embedding in the last estimate and recall that $s>\frac12$. 
Moreover, we find arguing as in \eqref{energy:diff.6} that
\begin{equation} \label{energy:diff.23}
|\widetilde{\mathcal{II}}|  \lesssim \|u\|_{H^s}\|w\|_{L^2}\|z\|_{L^2}+\|u\|_{H^s}^2\|w\|_{L^2}^2 .
\end{equation}

To deal with $\widetilde{\mathcal{III}}$, we use the second equation in \eqref{sbo:diff} and integrate by parts to obtain
\begin{equation*}
\widetilde{\mathcal{III}}=-\frac12 \int_{\mathbb{R}}D^s_xz D^s\partial_x(vz)\,d x+\mathrm{Re} \int_{\mathbb{R}}D_x^sz D^s_x\partial_x(\bar{u}w) \, dx=\widetilde{\mathcal{III}}_1+\widetilde{\mathcal{III}}_2 .
\end{equation*}
On the one hand, we decompose  $\widetilde{\mathcal{III}}_1$ as 
\begin{equation*}
\begin{split}
\widetilde{\mathcal{III}}_1&=-\frac12 \int_{\mathbb{R}}D^s_xz [D^s_x,z]\partial_xv\,d x-\frac12 \int_{\mathbb{R}}(D^s_xz) z D^s\partial_xv\,d x \\ & \quad -\frac12 \int_{\mathbb{R}}D^s_xz [D^s_x,v]\partial_xz\,d x-\frac12 \int_{\mathbb{R}}(D^s_xz) v D^s\partial_xz\,d x \\ 
&=\widetilde{\mathcal{III}}_{1,1}+\widetilde{\mathcal{III}}_{1,2}+\widetilde{\mathcal{III}}_{1,3}+\widetilde{\mathcal{III}}_{1,4}.
\end{split}
\end{equation*}
The H\"older inequality and Lemma \ref{kpcomm}  yield 
\begin{align}
\left| \widetilde{\mathcal{III}}_{1,1} \right|&\lesssim \|D^s_xz\|_{L^2} \left(\|\partial_xv\|_{L^{\infty}}\|D^sz\|_{L^2}+\|\partial_xz\|_{L^{\infty}}\|D^s_xv\|_{L^2} \right) ;  \label{energy:diff.24}\\ 
\left| \widetilde{\mathcal{III}}_{1,2} \right| &\lesssim \|D^s_xz\|_{L^2}\|z\|_{L^2}\|D^s_x\partial_xv\|_{L^{\infty}}; \label{energy:diff.25} \\ 
 \left| \widetilde{\mathcal{III}}_{1,3} \right| &\lesssim \|D^s_xz\|_{L^2} \left(\|\partial_xv\|_{L^{\infty}}\|D^sz\|_{L^2}+\|\partial_xz\|_{L^{\infty}}\|D^s_xv\|_{L^2} \right) ; \label{energy:diff.26} \\
 \left| \widetilde{\mathcal{III}}_{1,4} \right| &\lesssim \|\partial_xv\|_{L^{\infty}} \|D^s_xz\|_{L^2}^2.  \label{energy:diff.27}
\end{align}
On the other hand, $\widetilde{\mathcal{III}}_2$ cannot be estimated directly and will instead be cancelled out by a term coming from $\widetilde{\mathcal{V}}$. 

Next, we decompose $\widetilde{IV}$ as
\begin{displaymath}
\begin{split}
\widetilde{\mathcal{IV}}&=\mathrm{Im} \int_{\mathbb R} D^{s+\frac12}_x\overline{w}D^{s+\frac12}_x(vw) \, dx+\mathrm{Im} \int_{\mathbb R} D^{s+\frac12}_x\overline{w}D^{s+\frac12}_x(uz) \, dx  \\ &\quad+
\mathrm{Im} \int_{\mathbb R} D^{s+\frac12}_x\overline{w}D^{s+\frac12}_x\left(\left(\frac12|w|^2+|u|^2\right)w+\frac12u^2\overline{w}\right) \, dx \\ 
&=\widetilde{\mathcal{IV}}_1+\widetilde{\mathcal{IV}}_2+\widetilde{\mathcal{IV}}_3 . 
\end{split}
\end{displaymath}
Firstly, we observe from Lemma \ref{kpcomm} (ii) that 
\begin{equation}  \label{energy:diff.28} 
\begin{split}
\left| \widetilde{\mathcal{IV}}_1 \right| & = \left| \mathrm{Im} \int_{\mathbb R} D^{s+\frac12}_x\overline{w}[D^{s+\frac12}_x,v]w \, dx \right|  \\ & \lesssim \|D_x^{s+\frac12}w\|_{L^2} \left(\|D_x^{s+\frac12}v\|_{L^{\infty}}\|w\|_{L^2}+\|\partial_xv\|_{L^{\infty}}\|D^{s-\frac12}w\|_{L^2} \right).
\end{split}
\end{equation}
Secondly, we decompose $\widetilde{\mathcal{IV}}_2$ as 
\begin{equation*} 
\widetilde{\mathcal{IV}}_2=\mathrm{Im} \int_{\mathbb R} (D^{s+\frac12}_x\overline{w})uD^{s+\frac12}_xz \, dx+\mathrm{Im} \int_{\mathbb R} D^{s+\frac12}_x\overline{w}[D^{s+\frac12}_x,u]z \, dx = \widetilde{\mathcal{IV}}_{2,1}+\widetilde{\mathcal{IV}}_{2,2} .
\end{equation*}
While $\widetilde{\mathcal{IV}}_{2,1}$ will be canceled out by a term coming from $\widetilde{\mathcal{V}}$, we use Lemma \ref{kpcomm} (ii) and the Sobolev embedding (with $s>\frac12$) to get 
\begin{equation}  \label{energy:diff.29} 
\begin{split}
\left| \widetilde{\mathcal{IV}}_{2,2} \right| & \lesssim \|D_x^{s+\frac12}w\|_{L^2} \left( \|D_x^{s+\frac12}u\|_{L^2}\|z\|_{L^{\infty}}+\|\partial_xu\|_{L^{2_+}}\|D^{s-\frac12}z\|_{L^{\infty_-}} \right) \\ 
& \lesssim  \|D_x^{s+\frac12}u\|_{L^2}\|z\|_{H^s}\|D_x^{s+\frac12}w\|_{L^2} .
\end{split}
\end{equation}
Thirdly, using the fact that $H^{s+\frac12}$ is a Banach algebra, we deduce that 
\begin{equation}  \label{energy:diff.30}
\left| \widetilde{\mathcal{IV}}_{3} \right| \lesssim \left( \|D_x^{s+\frac12}u_1\|_{L^2}+ \|D_x^{s+\frac12}u_2\|_{L^2}\right)^2 \|D_x^{s+\frac12}w\|_{L^2}^2 .
\end{equation}

Now, we decompose $\widetilde{\mathcal{V}}$ as 
\begin{equation*} 
\begin{split}
 \widetilde{\mathcal{V}}&=\mathrm{Re} \int_{\mathbb{R}}D^{s-\frac12}_x\partial_tzD^{s-\frac12}_x\left(\overline{u}w\right)\, dx
 +\mathrm{Re} \int_{\mathbb{R}}D^{s-\frac12}_xzD^{s-\frac12}_x\left((\partial_t\overline{u})w\right)\, dx \\
 & \quad +\mathrm{Re} \int_{\mathbb{R}}D^{s-\frac12}_xzD^{s-\frac12}_x\left(\overline{u}\partial_tw\right)\, dx \\ 
 &=\widetilde{\mathcal{V}}_1+\widetilde{\mathcal{V}}_2+\widetilde{\mathcal{V}}_3 .
 \end{split}
\end{equation*}
Firstly, by using the second equation in \eqref{sbo:diff}, we write 
\begin{equation*} 
\begin{split}
 \widetilde{\mathcal{V}}_1&=\mathrm{Re} \int_{\mathbb{R}}D^{s-\frac12}_x\mathcal{H}\partial_x^2zD^{s-\frac12}_x\left(\overline{u}w\right)\, dx
-\frac12 \mathrm{Re} \int_{\mathbb{R}}D^{s-\frac12}_x\partial_x(vz)D^{s-\frac12}_x\left(\overline{u}w\right)\, dx
\\ 
 &=\widetilde{\mathcal{V}}_{1,1}+\widetilde{\mathcal{V}}_{1,2}  .
 \end{split}
\end{equation*}
On the one hand, by using $\mathcal{H}\partial_x=D^1_x$, we see that 
\begin{equation} \label{energy:diff.31}
 \widetilde{\mathcal{III}}_{2}+\widetilde{\mathcal{V}}_{1,1}=0 .
 \end{equation}
 On the other, by using that $H^s$ is a Banach algebra for $s>\frac12$, we have that 
 \begin{equation} \label{energy:diff.32}
\left|\widetilde{\mathcal{V}}_{1,1}\right| \lesssim \|vz\|_{H^s} \|\bar{u}w\|_{H^s} \lesssim   \|v\|_{H^s} \|u\|_{H^s}\|w\|_{H^s}\|z\|_{H^s} .
 \end{equation}
 Secondly, by using the first equation in \eqref{sbo}, we decompose ${\rm \widetilde{\mathcal{V}}_2}$ as
\begin{equation*}
\begin{split}
\widetilde{\mathcal{V}}_2&=\mathrm{Im} \int_{\mathbb{R}} D^{s-\frac12}_x z D^{s-\frac12}_x\left((\partial_x^2\overline{ u})w\right)\,dx
-\mathrm{Im} \int_{\mathbb{R}} D^{s-\frac12}_x zD^{s-\frac12}_x\left(\left( \overline{u_1} v_1 +\overline{u_2} v_2\right)w\right)\,dx \\ &
\quad -\mathrm{Im} \int_{\mathbb{R}} D^{s-\frac12}_x zD^{s-\frac12}_x\left(\left( |u_1|^2\overline{u_1}+|u_2|^2\overline{u_2} \right)w\right)\,dx \\ 
&=\widetilde{\mathcal{V}}_{2,1}+\widetilde{\mathcal{V}}_{2,2}+\widetilde{\mathcal{V}}_{2,3} .
\end{split}
\end{equation*}
By integrating by parts and using the identity $\partial_x=-\mathcal{H}D^1_x$, we rewrite $\widetilde{\mathcal{V}}_{2,1}$ as 
\begin{equation*}
\begin{split}
\widetilde{\mathcal{V}}_{2,1}&=-\mathrm{Im} \int_{\mathbb{R}} D^{s}_x z \mathcal{H}D^{s}_x\left((\partial_x\overline{ u})w\right)\,dx-\mathrm{Im} \int_{\mathbb{R}} D^{s-\frac12}_x z D^{s-\frac12}_x\left((\partial_x\overline{ u})(\partial_xw)\right)\,dx \\ 
&=\widetilde{\mathcal{V}}_{2,1,1}+\widetilde{\mathcal{V}}_{2,1,2}.
\end{split}
\end{equation*}
The contribution $\widetilde{\mathcal{V}}_{2,1,2}$ will be canceled out by a term coming from $\widetilde{\mathcal{V}}_{3}$. To handle the contribution $\widetilde{\mathcal{V}}_{2,1,1}$, we use $\partial_x=-\mathcal{H}D^{\frac12}_xD^{\frac12}_x$ and we decompose it further as 
\begin{equation*} 
\begin{split}
\widetilde{\mathcal{V}}_{2,1,1}&=-\mathrm{Im} \int_{\mathbb{R}} D^{s}_x zD^{s+\frac12}_x\left((D_x^{\frac12}\overline{ u})w\right)\,dx-\mathrm{Im} \int_{\mathbb{R}} D^{s}_x z\mathcal{H}D^{s}_x[\mathcal{H}D_x^{\frac12},w]D_x^{\frac12}\overline{ u}\,dx\\
&=-\mathrm{Im} \int_{\mathbb{R}} (D^{s}_x z) w D^{s+1}_x\overline{ u}\,dx
-\mathrm{Im} \int_{\mathbb{R}} D^{s}_x z[D^{s+\frac12}_x,w]D_x^{\frac12}\overline{ u}\,dx
 \\ 
&\quad -\mathrm{Im} \int_{\mathbb{R}} D^{s}_x z\mathcal{H}[\mathcal{H}D_x^{s+\frac12},w]D_x^{\frac12}\overline{ u}\,dx-\mathrm{Im} \int_{\mathbb{R}} D^{s}_x z\mathcal{H}[D_x^{s},w]\partial_x\overline{ u}\,dx \\ 
&=\widetilde{\mathcal{V}}_{2,1,1,1}+\widetilde{\mathcal{V}}_{2,1,1,2}+\widetilde{\mathcal{V}}_{2,1,1,3}+\widetilde{\mathcal{V}}_{2,1,1,4}. 
\end{split}
\end{equation*}
On the one hand, it follows from H\"older's inequality that 
\begin{equation}  \label{energy:diff.33}
\left| \widetilde{\mathcal{V}}_{2,1,1,1} \right|  \lesssim  \|D_x^sz\|_{L^2}\|w\|_{L^2}\|D^{s+1}u\|_{L^{\infty}} .
\end{equation}
On the other hand,  Lemma \ref{kpcomm} that 
\begin{equation}  \label{energy:diff.34}
\begin{split}
\left| \widetilde{\mathcal{V}}_{2,1,1,2} \right| +\left| \widetilde{\mathcal{V}}_{2,1,1,3} \right| &\lesssim \|D_x^sz\|_{L^2} \big( \|D_x^{s+\frac12}w\|_{L^2} \|D_x^{\frac12}u\|_{L^{\infty}}+ \|\partial_xw\|_{L^{2_+}} \|D_x^{s}u\|_{L^{\infty_-}}\big) \\ &
\lesssim  \|D_x^sz\|_{L^2}  \|u\|_{H^{s+\frac12}} \|w\|_{H^{s+\frac12}}  
\end{split}
\end{equation}
and 
\begin{equation}  \label{energy:diff.35}
\begin{split}
\left| \widetilde{\mathcal{V}}_{2,1,1,4} \right| &\lesssim \|D_x^sz\|_{L^2} \big( \|D_x^{s}w\|_{L^{\infty_-}} \|\partial_xu\|_{L^{2_+}}+ \|\partial_xw\|_{L^{2_+}} \|D_x^{s}u\|_{L^{\infty_-}}\big) \\ &
\lesssim  \|D_x^sz\|_{L^2}  \|u\|_{H^{s+\frac12}} \|w\|_{H^{s+\frac12}}  .
\end{split}
\end{equation}
Moreover, by using that $H^s$ is a Banach algebra for $s>\frac12$, we deduce that 
\begin{equation}  \label{energy:diff.36}
\begin{split}
\left| \widetilde{\mathcal{V}}_{2,2} \right| &\lesssim \|D_x^{s-\frac12}z\|_{L^2} \|\left( \overline{u_1} v_1 +\overline{u_2} v_2\right)w \|_{H^s} \\ &
\lesssim \left( \|u_1\|_{H^s}  \|v_1\|_{H^{s}}+\|u_2\|_{H^s}  \|v_2\|_{H^{s}} \right) \|z\|_{H^{s}} \|w\|_{H^{s}}  
\end{split}
\end{equation}
and 
\begin{equation}  \label{energy:diff.37}
\begin{split}
\left| \widetilde{\mathcal{V}}_{2,3} \right| &\lesssim \|D_x^{s-\frac12}z\|_{L^2} \|\left( |u_1|^2\overline{u_1}+|u_2|^2\overline{u_2} \right)w \|_{H^s} \\ &
\lesssim \left( \|u_1\|_{H^s}^3+\|u_2\|_{H^s}^3 \right) \|z\|_{H^{s}} \|w\|_{H^{s}}  .
\end{split}
\end{equation}

Finally, by using the first equation in \eqref{sbo:diff}, we decompose $ \widetilde{\mathcal{V}}_3$ as 
\begin{equation*} 
\begin{split}
 \widetilde{\mathcal{V}}_3&=-\mathrm{Im} \int_{\mathbb{R}}D^{s-\frac12}_xzD^{s-\frac12}_x\left(\overline{u}\partial_x^2w\right)\, dx
 +\frac12 \mathrm{Im} \int_{\mathbb{R}}D^{s-\frac12}_xzD^{s-\frac12}_x\left(\overline{u}(vw+uz)\right)\, dx \\ 
 & \quad + \mathrm{Im} \int_{\mathbb{R}}D^{s-\frac12}_xzD^{s-\frac12}_x\left(\overline{u}\big(\frac14|w|^2w+\frac12|u|^2w+\frac14 u^2\overline{w}\big)\right)\, dx\\
 &=\widetilde{\mathcal{V}}_{3,1}+\widetilde{\mathcal{V}}_{3,2}+\widetilde{\mathcal{V}}_{3,3} .
 \end{split}
\end{equation*}
By integrating by parts, we further decompose $\widetilde{\mathcal{V}}_{3,1}$ as
\begin{equation*} 
\begin{split}
 \widetilde{\mathcal{V}}_{3,1}&=\mathrm{Im} \int_{\mathbb{R}}D^{s-\frac12}_xzD^{s-\frac12}_x\left(\partial_x\overline{u}\partial_xw\right)\, dx
 +\mathrm{Im} \int_{\mathbb{R}}D^{s-\frac12}_x\partial_xzD^{s-\frac12}_x\left(\overline{u}\partial_xw\right)\, dx \\ 
 &=\widetilde{\mathcal{V}}_{3,1,1}+\widetilde{\mathcal{V}}_{3,1,2} .
 \end{split}
\end{equation*}
On the one hand, we have the cancellation 
\begin{equation} \label{energy:diff.38}
 \widetilde{\mathcal{V}}_{2,1,2}+\widetilde{\mathcal{V}}_{3,1,1}=0 .
 \end{equation}
 On the other hand, by using the identity $\partial_x=-\mathcal{H}D_x^1$, we rewrite $\widetilde{\mathcal{V}}_{3,1,2}$ as
\begin{equation*} 
\begin{split}
 \widetilde{\mathcal{V}}_{3,1,2}&=-\mathrm{Im} \int_{\mathbb{R}}\mathcal{H}D^{s+\frac12}_xzD^{s-\frac12}_x\left(\overline{u}\partial_xw\right)\, dx
  \\ &=\mathrm{Im} \int_{\mathbb{R}}(D^{s+\frac12}_xz) \overline{u} D^{s+\frac12}_xw\, dx+
  \mathrm{Im} \int_{\mathbb{R}}(D^{s+\frac12}_xz) [\mathcal{H}D_x^{s-\frac12},\overline{u}] \partial_xw\, dx\\
 &=\widetilde{\mathcal{V}}_{3,1,2,1}+\widetilde{\mathcal{V}}_{3,1,2,2} 
 \end{split}
\end{equation*}
and we use the cancellation
\begin{equation} \label{energy:diff.39}
 \widetilde{\mathcal{IV}}_{2,1}+\widetilde{\mathcal{V}}_{3,1,2,1}=0 
 \end{equation}
 to deal with the first term. To handle the second term, we rewrite it as 
\begin{equation*} 
\begin{split}
 \widetilde{\mathcal{V}}_{3,1,2,2}&=\mathrm{Im} \int_{\mathbb{R}}D^{s}_xz[\mathcal{H}D^{s}_x,\overline{u}]\partial_xw\, dx 
 -\mathrm{Im} \int_{\mathbb{R}}D^{s}_xz[D^{\frac12}_x,\overline{u}]D_x^{s+\frac12}w\, dx\\
   &=\widetilde{\mathcal{V}}_{3,1,2,2,1}+\widetilde{\mathcal{V}}_{3,1,2,2,2} , 
 \end{split}
\end{equation*}
and we deduce from Lemma \ref{kpcomm} that 
\begin{equation}  \label{energy:diff.40}
\begin{split}
\left| \widetilde{\mathcal{V}}_{3,1,2,2,1} \right| &\lesssim \|D_x^sz\|_{L^2} \big( \|D_x^{s}u\|_{L^{\infty_-}} \|\partial_xw\|_{L^{2_+}}+ \|\partial_xu\|_{L^{2_+}} \|D_x^{s}w\|_{L^{\infty_-}}\big) \\ &
\lesssim  \|D_x^sz\|_{L^2}  \|u\|_{H^{s+\frac12}} \|w\|_{H^{s+\frac12}}  
\end{split}
\end{equation}
and
\begin{equation}  \label{energy:diff.41}
\begin{split}
\left| \widetilde{\mathcal{V}}_{3,1,2,2,2} \right| &\lesssim \|D_x^sz\|_{L^2}  \|D_x^{\frac12}u\|_{L^{\infty}} \|D^{s+\frac12}w\|_{L^{2}}
\lesssim  \|D_x^sz\|_{L^2}  \|u\|_{H^{s+\frac12}} \|w\|_{H^{s+\frac12}}  .
\end{split}
\end{equation} 
Moreover, by using that $H^s$ is a Banach algebra for $s>\frac12$, we deduce that 
\begin{equation}  \label{energy:diff.42}
\begin{split}
\left| \widetilde{\mathcal{V}}_{3,2} \right| &\lesssim \|D_x^{s-\frac12}z\|_{L^2} \|\left(\overline{u}(vw+uz\right)\|_{H^s} \\ &
\lesssim \|u\|_{H^s}  \|v\|_{H^{s}}  \|z\|_{H^{s}} \|w\|_{H^{s}}+\|u\|_{H^s}^2 \|z\|_{H^{s}}^2   
\end{split}
\end{equation}
and 
\begin{equation}  \label{energy:diff.43}
\begin{split}
\left| \widetilde{\mathcal{V}}_{3,3} \right| &\lesssim \|D_x^{s-\frac12}z\|_{L^2} \|\left(\overline{u}\big(\frac14|w|^2w+\frac12|u|^2w+\frac14 u^2\overline{w}\big)\right)\|_{H^s} \\ &
\lesssim \left( \|u_1\|_{H^s}+\|u_2\|_{H^s} \right)^3 \|z\|_{H^{s}} \|w\|_{H^{s}}  .
\end{split}
\end{equation}

Therefore, we conclude the proof of estimate \eqref{m-energy:diff.3} by gathering \eqref{energy:diff.21}-\eqref{energy:diff.43}.
This finishes the proof of Proposition \ref{en_est:diff}.
\end{proof}

\section{Proof of Theorem \ref{teo1}}

\subsection{Refined Strichartz estimates}

%
%
%
%

One of the main ingredients in our analysis is a  refined Strichartz estimates for solutions
of the linear non-homogeneous Benjamin-Ono equation.  This estimate is proved by Kenig and Koenig in Proposition 2.8 in \cite{kk} and is based on previous ideas by Koch and Tzvetkov in \cite{kt}. 
\begin{prop}\label{propo1}
	Let   $s>\frac12$, $\delta \in[0,1]$ and $0<T \le 1$. Assume that $v \in C([0,T]: H^s(\mathbb{R}))$  is a solution  to the equation
	\begin{equation}\label{eq3}
		\partial_{t}v-\mathcal{H}\partial_{x}^{2}v= F .
	\end{equation}
Then, for any $\epsilon>0,$
\begin{equation}\label{e5}
\begin{split}
\|\partial_{x}v\|_{L^{2}_{T}L^{\infty}_{x}}
&\lesssim T^{\frac12}\|J_{x}^{1+\frac{\delta}4+\epsilon} v\|_{L^{\infty}_{T}L^{2}_{x}}
+ \|J_{x}^{1-\frac{3\delta}{4}+\epsilon} F\|_{L^{2}_{T}L^{2}_{x}} .
\end{split}
\end{equation} 
\end{prop}

By relying on this estimate, we can control the term $\|\partial_xv\|_{L^1_TL^{\infty}_x}$ appearing in the energy estimates. 
\begin{lem} \label{est:dxv}
Let $s>\frac54$ and $0<T \le 1$. There exists $\kappa_{2,s}>0$ such that for any solution $(u,v) \in C([0,T]: H^{s+\frac12}(\mathbb{R}) \times H^s(\mathbb{R}))$ of \eqref{sbo}, we have 
\begin{equation} \label{est:dxv.1}
\|\partial_xv\|_{L^1_TL^{\infty}_x} \le \kappa_{2,s} T \left(\|v\|_{L^{\infty}_TH^s_x}+\|v\|_{L^{\infty}_TH^s_x}^2+\|u\|_{L^{\infty}_TH^s_x}^2 \right) .
\end{equation}
\end{lem}

\begin{proof}
The proof of estimate \eqref{est:dxv.1} follows directly by applying the H\"older inequality in time, estimate \eqref{e5} with $\delta=1$ and $F=-\frac12\partial_x(v^2)+\partial_x(|u|^2)$, and the fact that $H^s(\mathbb R)$ is a Banach algebra for $s>\frac12$.
\end{proof}

\subsection{Well-posedness for smooth initial data}
As far as we know, there does not exist a well-posedness theory for the system \eqref{sbo} when $\rho \neq 0$. The next result is based on the energy estimates derived in Section \ref{Sec:EnEs}. 

\begin{thm}\label{lwpsmoothsol}
Let $s>\frac32$. Then, for any $(u_0, v_0)\in H^{s+\frac12}(\R)\times H^s(\R)$,  there exist a positive time 
$T=T(\|(u_0, v_0)\|_{H^{s+\frac12}\times H^s})$ and a unique maximal solution $(u,v)$ of the IVP \eqref{sbo} in $C\big([0,T^ *): H^{s+\frac12}(\R)\times H^s(\R) \big)$ with $T^*>T(\|(u_0, v_0)\|_{H^{s+\frac12}\times H^s})$. 
If the maximal time of existence $T^*$ is finite, then 
\begin{equation*}
\lim_{t \nearrow T^*} \|(u(t),v(t))\|_{H^{s+\frac12} \times H^s}=+\infty.
\end{equation*}

Moreover, for any $0<T'<T$, there exists a neighborhood $\mathcal{U}$ of $(u_0,v_0)$ in $H^{s+\frac12}(\mathbb R) \times H^s(\mathbb R)$ such that the flow map data-to-solution 
\begin{equation*}
S: \mathcal{U} \rightarrow C\big([0,T]: H^{s+\frac12}_{x}(\mathbb{R})\times H^{s}_{x}(\mathbb{R})\big) \, , (\tilde{u}_0,\tilde{v}_0) \mapsto (\tilde{u},\tilde{v})
\end{equation*}
is continous. \end{thm}

\begin{proof}
First observe that by using the rescaled version \eqref{sbo_rescaled} of the system, we can assume by choosing $\lambda$ small enough that the norm initial datum $\|(u_0, v_0)\|_{H^{s+\frac12}\times H^s}$ is small.

Then, the proof of the existence and the uniqueness is a combination of the parabolic regularization method with the energy estimates obtained in Propositions \ref{import} and \ref{en_est:diff} and taking into account Remarks \ref{rem:en_est} and \ref{rem:en_est_diff}. We refer to the proof of Theorem 1.6 in \cite{Pa} for the details in a similar setting. 

The continuous dependence and persistence is obtained by applying the Bona-Smith approximation method. We refer to \cite{bona}, \cite{Iorio}, \cite{LiPo} and \cite{Pa} for the details.
\end{proof}

\subsection{A priori estimates}

Let $(u_0,v_0) \in H^{\infty}(\R)\times H^{\infty}(\R)$. From the above result there exists a solution $(u,v) \in C\big([0,T^ *): H^{\infty}(\R)\times H^{\infty}(\R) \big)$, where $T^*$ is the maximal time of existence satisfying $T^* \ge  T(\|(u_0, v_0)\|_{H^{\frac52}\times H^2})$. Moreover, we have the blow-up alternative 
\begin{equation} \label{blow-up_alt}
\lim_{t \nearrow T^*} \|(u(t),v(t))\|_{H^{\frac52} \times H^2}=+\infty  \quad \text{if} \quad T^* <+\infty .
\end{equation}

Let $\frac54<s \le \frac32$. By using a bootstrap argument, we prove that the solution $(u,v)$ satisfies a suitable \textit{a priori} estimate on a positive time interval depending only on the $H^{s+\frac12} \times H^s$ norm of the initial datum.

\begin{lem} \label{apriori} 
Let $\;\frac54<s \le \frac32$. There exist positive constant $\delta_s$, $K_s$ and $A_s$ such that if 
$$ 
\|(u_0, v_0)\|_{H^{s+\frac12}\times H^s} \le \delta_s,
$$
then $T^* >T_s:=(A_s(1+\|(u_0, v_0)\|_{H^{s+\frac12}\times H^s}))^{-2}$, 
\begin{equation} \label{apriori.1}
\|\left(u,v\right)\|_{L^{\infty}_{T_s}(H^{s+\frac{1}{2}}\times H^{s})_{x}} \le 8 \|(u_0, v_0)\|_{H^{s+\frac12}\times H^s} \quad \text{and} \quad \|\partial_xv\|_{L^1_{T_s}L^{\infty}_x} \le K_s .
\end{equation}
\end{lem} 

\begin{proof} Let $\frac54<s \le \frac32$. We set $\delta_s:= 2^{-6}\min \{ c_s^{-1},\widetilde{c}_s^{-1} \}$, 
where $c_s$ and $\widetilde{c}_s$ are respectively defined in Propositions \ref{import} and \ref{en_est:diff}. 
Assume that $\|(u_0, v_0)\|_{H^{s+\frac12}\times H^s} \le \delta_s$. Then, we define 
\begin{equation*}
\tilde{T}_s:= \sup \left\{ T \in (0,T^*) :  \|\left(u,v\right)\|_{L^{\infty}_{T}(H^{s+\frac{1}{2}}\times H^{s})_{x}} \le 8 \|(u_0, v_0)\|_{H^{s+\frac12}\times H^s}  \right\} .
\end{equation*}
Note that the above set is nonempty since $(u,v) \in C\big([0,T^ *): H^{\infty}(\R)\times H^{\infty}(\R) \big)$, so that $\tilde{T}_s$ is well-defined. We argue by contradiction and assume that \[0<\tilde{T}_s<(A_s(1+\|(u_0, v_0)\|_{H^{s+\frac12}\times H^s}))^{-2} \le 1\] for $A_s=2^6 (\log 2)^{-1}(1+\kappa_{1,s}+\kappa_{1,2})(1+\kappa_{2,s})$, where $\kappa_{1,s}$ and $\kappa_{2,s}$ are respectively defined in Proposition \ref{import} and Lemma \ref{est:dxv}.

Let $0<T<\tilde{T}_s$. We have from the definition of $\tilde{T}_s$ that 
\[ \|\left(u,v\right)\|_{L^{\infty}_{T}(H^{s+\frac{1}{2}}\times H^{s})_{x}} \le 8 \|(u_0, v_0)\|_{H^{s+\frac12}\times H^s} \le \frac18 c_s^{-1} .\] 
Then, estimate \eqref{est:dxv.1} yields
\begin{equation*}
\|\partial_xv\|_{L^1_TL^{\infty}_x}  \le 8\kappa_{2,s} T \left(1+16 \|(u_0, v_0)\|_{H^{s+\frac12}\times H^s} \right)  \|(u_0, v_0)\|_{H^{s+\frac12}\times H^s} \le \frac{\log 2}{8(1+\kappa_{1,s}+\kappa_{1,2})} .
\end{equation*}
Hence, we deduce by using the energy estimate \eqref{m-energy3} at the level $s=2$ that 
\begin{equation*}
\|\left(u,v\right)\|_{L^{\infty}_{T}(H^{\frac{5}{2}}\times H^{2})_{x}} \le 4  \|(u_0, v_0)\|_{H^{\frac52}\times H^2}, \quad \forall 0<T<\tilde{T}_s .
 \end{equation*}
 This implies in view of the blow-up alternative \eqref{blow-up_alt} that $\tilde{T}_s<T^*$. 
 
 Now, applying again the energy estimate \eqref{m-energy3} at the level $s$ implies that 
 $$
 \|\left(u,v\right)\|_{L^{\infty}_{\tilde{T}_s}(H^{s+\frac{1}{2}}\times H^{s})_{x}} \le 4  \|(u_0, v_0)\|_{H^{s+\frac12}\times H^s}\,
 $$
 so that by continuity, there exists some $T^{\dagger}_s$ satisfying $\tilde{T}_s<T^{\dagger}_s<T^*$ such that 
 $$
 \|\left(u,v\right)\|_{L^{\infty}_{T^{\dagger}_s}(H^{s+\frac{1}{2}}\times H^{s})_{x}} \le 6  \|(u_0, v_0)\|_{H^{s+\frac12}\times H^s}.
 $$
This contradicts the definition of $\tilde{T}_s$. 
 
 Therefore, $\tilde{T}_s \ge T_s:=(A_s(1+\|(u_0, v_0)\|_{H^{s+\frac12}\times H^s}))^{-2}$ and we argue as above to get the bound for $\|\partial_xv\|_{L^1_{T_s}L^{\infty}_x}$. This concludes the proof of the lemma.
\end{proof}

\subsection{Uniqueness and $H^{\frac12}\times L^2$-Lipschitz bounds for the flow} \label{section_uniq}
Let $s>\frac54$ and let $(u_1,v_1)$ and $(u_2,v_2)$ be two solutions of \eqref{sbo} in the class \eqref{clas1}-\eqref{clas2} corresponding to initial data $(u_1^0,v_1^0)$ and $(u_2^0,v_2^0)$. By using the change of variables \eqref{sbo_rescaled}, we can always assume that $\|\left(u_i^0,v_i^0\right)\|_{H^{s+\frac{1}{2}}_{x}\times H^{s}_{x}}\le \frac1{2\widetilde{c}_s}$, for $i=1,2$, where $\widetilde{c}_s$ is the positive constant given in Proposition \ref{en_est:diff}. Then, by possibly restricting the time interval, we can assume that \eqref{m-energy:diff.0} holds on $[0,T]$. We define the positive number 
\begin{equation*} 
K:= \max \left\{ \|\partial_xv_1\|_{L^1_TL^{\infty}_x}, \|\partial_xv_2\|_{L^1_TL^{\infty}_x} \right\} .
\end{equation*}

Therefore, we deduce from \eqref{m-energy:diff.1}, \eqref{m-energy:diff.2} and the Gronwall inequality that
\begin{equation} \label{Lip_bound_flow}
 \|(u_1-u_2,v_1-v_2)\|_{L^{\infty}_T(H^{\frac12}_{x}\times L^2_{x})} \le 2\,e^{c(K+1)T}\|(u_1^0-u_2^0,v_1^0-v_2^0)\|_{H^{\frac12}\times L^2} .
\end{equation}
Estimate \eqref{Lip_bound_flow} provides the uniqueness result in Theorem \ref{teo1}. by choosing $(u_1^0,v_1^0)=(u_2^0,v_2^0)=(u_0,v_0)$. 


\subsection{Existence, persistence and continuous dependence}

Let $\frac54<s \le \frac32$ and let $(u_0,v_0) \in H^{s+\frac12}(\mathbb R) \times H^s(\mathbb R)$. By using the change of variables \eqref{sbo_rescaled}, we can always assume that \newline $\|(u_0, v_0)\|_{H^{s+\frac12}\times H^s} \le \delta_s$, where $\delta_s$ is the positive constant given by Lemma \ref{apriori}.

We regularize the initial datum as follows. Let $\chi$ be a smooth cutoff function satisfying
\begin{equation}\label{chi}
\chi \in C_0^{\infty}(\mathbb R), \quad 0 \le \chi \le 1, \quad
\chi_{|_{[-1,1]}}=1 \quad \mbox{and} \quad  \mbox{supp}(\chi)
\subset [-2,2].
\end{equation}
Then we define
\begin{equation*}
(u_{0,n},v_{0,n})= (P_{\le n}u_0,P_{\le n}v_0)=\left(\left(\chi({|\xi|}/{n}) \widehat{u}_0(\xi)\right)^{\vee},\left(\chi({|\xi|}/{n}) \widehat{v}_0(\xi)\right)^{\vee}  \right)\, ,
\end{equation*}
for any $n \in \mathbb N$, $n \ge 1$.

Then, the following estimates are well-known (see for example Lemma 5.4 in \cite{lps}).
\begin{lem} \label{BSreg}\hskip10pt

\begin{itemize}
\item[(i)]
Let $\sigma \ge 0$ and $n \ge 1$. Then,
\begin{equation} \label{BSreg.1}
\|u_{0,n}\|_{H^{s+\frac12+\sigma}} \lesssim n^{\sigma} \|u_0\|_{H^{s+\frac12}} \quad \text{and} \quad  \|v_{0,n}\|_{H^{s+\sigma}} \lesssim n^{\sigma} \|v_0\|_{H^s} .
\end{equation}

\item[(ii)]
Let $0 \le \sigma \le s$ and $m\ge n\ge 1$. Then, 
\begin{equation} \label{BSreg.4}
\|u_{0,n}-u_{0,m}\|_{H^{s+\frac12-\sigma}} \underset{n \to +\infty}{=}o(n^{-\sigma}) \quad \text{and} \quad \|v_{0,n}-v_{0,m}\|_{H^{s-\sigma}} \underset{n \to +\infty}{=}o(n^{-\sigma}) .
\end{equation}
\end{itemize}
\end{lem}

Now, for each $n \in \mathbb N$, $n \ge 1$, we consider the solution $(u_n,v_n)$ to \eqref{sbo} emanating from $(u_{0,n},v_{0,n})$ defined on their maximal time interval $[0,T^{\star}_n)$. From Lemmas \ref{apriori} and \ref{BSreg} (i) with $\sigma=0$, there exists a positive time 
\begin{equation} \label{defT}
T:=(A_s(1+\|(u_0, v_0)\|_{H^{s+\frac12}\times H^s}))^{-2} \, ,
\end{equation} 
(where $A_s$ is a positive constant), independent of $n$, such that $(u_n,v_n) \in C([0,T] : H^{\infty}(\mathbb R)\times H^{\infty}(\mathbb R))$ is defined on the time interval $[0,T]$ and satisfies 
\begin{equation} \label{existence.1}
\|\left(u_n,v_n\right)\|_{L^{\infty}_{T}(H^{s+\frac{1}{2}}\times H^{s})_{x}} \le 8 \|(u_0, v_0)\|_{H^{s+\frac12}\times H^s}\end{equation}
and 
\begin{equation} \label{existence.2}
K:=\sup_{n \ge 1}\big\{ \|\partial_x v_n\|_{L^1_TL^{\infty}_{x}} \big\} <+\infty\, .
\end{equation}

Let $m \ge n \ge 1$. We set $w_{n,m} := u_n-u_m$ and $z_{n,m}:= v_n-v_m$. Then, $(w_{n,m},z_{n,m})$ satisfies \eqref{sbo:diff} with initial datum $(w_{n,m}(\cdot,0),z_{n,m}(\cdot,0))=(u_{0,n}-u_{0,m},v_{0,n}-v_{0,m})$. Then, by using  \eqref{Lip_bound_flow} and \eqref{BSreg.4} with $\sigma=s$, we deduce that
\begin{equation} \label{existence.3}
 \|(w_{n,m},z_{n,m})\|_{L^{\infty}_T(H^{\frac12}_{x}\times L^2_{x})} \le 2\,e^{c(K+1)T}\|(u_{0,n}-u_{0,m},v_{0,n}-v_{0,m})\|_{H^{\frac12}\times L^2}  \underset{n \to +\infty}{=} o(n^{-s}) ,
\end{equation}
which implies interpolating with \eqref{existence.1} that 
\begin{equation*} 
\begin{split}
\|(w_{n,m},z_{n,m})\|_{L^{\infty}_T(H^{\sigma+\frac12}_{x}\times H^{\sigma}_{x})} &\le \|(w_{n,m},z_{n,m})\|_{L^{\infty}_T(H^{s+\frac12}_{x}\times H^{s}_{x})}^{\frac{\sigma}s} \|(w_{n,m},z_{n,m})\|_{L^{\infty}_T(H^{\frac12}_{x}\times L^2_{x})}^{1-\frac{\sigma}s} \\ &\underset{n \to +\infty}{=}o(n^{-(s-\sigma)}) \, ,
\end{split}
\end{equation*}
for all $0 \le \sigma <s$.

Therefore, we deduce  that $\{(u_n,v_n) \}$ is a Cauchy sequence in $L^{\infty}([0,T] : H^{\sigma+\frac12}(\mathbb R)\times H^{\sigma}(\mathbb R))$, for any $0 \le \sigma<s$. Hence, it is not difficult to verify passing to the limit as $n \to +\infty$ that $(u,v) =\lim_{n \to +\infty}(u_n,v_n)$ is a weak solution to \eqref{sbo} in the class $C([0,T] : H^{\sigma+\frac12}(\mathbb R)\times H^{\sigma}(\mathbb R)) $ and satisfying $\|\partial_xv\|_{L^1_TL^{\infty}_x} \le K$, for any $0 \le \sigma<s$. 

Finally, the proof that $u$ belongs to the class \eqref{clas1} and of the continuous dependence of the flow follows from the classical Bona-Smith argument \cite{bona}. The proof relies on the energy estimate \eqref{m-energy:diff.3} and we refer the readers to \cite{lps} for more details in this setting. 

\vspace{0.5cm}
\noindent{\bf Acknowledgements.}  The authors are grateful to the Mathematics Department of Bergen University and to the Instituto de Matem\'atica Pura e Applicada where part of this work was done. F.L. was partially supported by CNPq grant 305791/2018-4 and FAPERJ grant E-26/202.638/2019 and MathAmSud EEQUADD II. D.P. was supported by a Trond Mohn Foundation grant. They authors would like to thank the anonymous referees for their careful proofreading of the manuscript.

\end{document}